\documentclass{amsart}
\pdfoutput=1
\usepackage{amssymb,latexsym,amsmath,amscd,graphicx,graphics,epic,eepic,enumerate, amsfonts, amssymb, mathrsfs, pifont}

\theoremstyle{plain}
\newtheorem{theorem}{Theorem}[section]
\newtheorem{lemma}[theorem]{Lemma}

\theoremstyle{definition}

\newtheorem{definition}[theorem]{Definition}

\theoremstyle{remark}
\newtheorem{remark}[theorem]{Remark}

\def \R{\mathbb{R}}
\def \Z{\mathbb{Z}}

\begin{document}

\title{A Homological Approach to Relative Knot Invariants}

\author{Georgi D. Gospodinov}

\address{Franklin W. Olin College of Engineering, 1000 Olin Way, Needham, MA 02492}

\email{georgi.gospodinov@olin.edu}


\begin{abstract}
We define relative versions of the classical invariants of Legendrian and transverse knots in contact 3-manifolds for knots that are homologous to a fixed reference knot. We show these invariants are well-defined and give some basic properties.
\end{abstract}

\maketitle

\section{Introduction}

A smooth oriented knot (an embedding of $S^1$) in a contact 3-manifold $(M,\xi)$ is called \textit{Legendrian} if it is everywhere tangent to the contact planes $\xi$ and {\em transverse} if it is everywhere transverse to $\xi$. Null-homologous Legendrian knots in a closed contact 3-manifold have two classical invariants, the Thurston-Bennequin invariant and the rotation number. Null-homologous transverse knots in a contact 3-manifold have a classical invariant, the self-linking number.

Our goal is to generalize these ideas to the case when $K$ is not null-homologous in $M$ (that is, $[K]\neq 0\in H_1(M;\Z)$), in particular, to the case when $K$ is homologous to another fixed knot $J$. We prove the following theorem. 

\begin{theorem}\label{main}
Consider homologous knots $K$ and $J$ in a contact 3-manifold $(M,\xi)$ with $K\cup J=\partial\Sigma$ for an embedded oriented surface $\Sigma$ and orient $K$ and $J$ as boundary components of $\Sigma$ so that $[\partial\Sigma]=[K]-[J]$.
\begin{enumerate}[(a)]
\item If $K$ and $J$ are Legendrian, then the relative Thurston-Bennequin number of $K$ relative to $J$ is well-defined.
\item If $K$ and $J$ are Legendrian, then the relative rotation number of $K$ relative to $J$ is well-defined up to an integer $d$ such that $d=c_1(\xi)([\Sigma])$, where $[\Sigma]\in H_2(M,K\cup J)$. For a tight contact 3-manifold $(M,\xi)$, this integer ambiguity vanishes so the relative rotation number is well-defined.
\item If $K$ and $J$ are transverse, then the relative self-linking number of $K$ relative to $J$ is well-defined up to an integer $d$ such that  $d=c_1(\xi)([\Sigma])$, where $[\Sigma]\in H_2(M,K\cup J)$. For a tight contact 3-manifold $(M,\xi)$, this integer ambiguity vanishes so the relative self-linking number is well-defined.
\end{enumerate}
\end{theorem}

In order to prove Theorem \ref{main}, we first discuss the dependence of each relative invariant on the Seifert surface for $K$ and $J$. Then we show that although Legendrian (resp. transverse) isotopies do not preserve the homology type of $K$ in $H_1(M\setminus J)$, they do not change the values of the respective relative invariants. 

\begin{remark} 
Since the contact structure $\xi$ is trivial over the compact Seifert surfaces with boundary, we do not require that $\xi$ be globally trivializable over $M$.
\end{remark}

\section{Acknowledgements} 

I would like to express deep gratitude to John Etnyre for his guidance and help with many fundamental and technical aspects of this work. I am also grateful to my advisor Danny Ruberman for his patience and support throughout my graduate years when the ideas of this paper were developed. I would like to also thank Vladimir Chernov and Jeremey Van Horn-Morris for helpful discussions and correspondence. 

\section{Background}

If $K\subset (M,\xi)$ is a null-homologous Legendrian knot in a contact 3-manifold, then from the exact sequence for the relative homology of the pair $(M,K)$, we see that there exists a relative class $\alpha\in H_2(M,K;\Z)$ such that $\alpha$ maps to $[K]\in H_1(K;\Z)$ under the boundary map $\partial:H_2(M,K;\Z)\rightarrow H_1(K;\Z)$, moreover, there exists a smooth embedded oriented Seifert surface $\Sigma\subset M$ for $K$ such that $[\Sigma]=\alpha$. 

If $K\subset (M,\xi)$ is a null-homologous Legendrian knot in a closed oriented contact 3-manifold, then the {\em Thurston-Bennequin invariant} of $K$ is defined as $tb_\Sigma(K)=tw_K(\xi,Fr_\Sigma)$, which measures the number of $2\pi$-twists (with sign) of the contact framing with respect to the Seifert framing along $K$. 

Equivalently, one can define $tb_\Sigma(K)$ as $lk(K,K')=K'\cdot\Sigma$, where $K'$ is a push-off of $K$ in the direction normal to the contact planes along $K$ (since $\xi$ is trivial over $\Sigma$ and, in particular, along $K=\partial\Sigma$, see \cite{etnyre:intro, etnyre:knots}, it is coorientable along $K$). The linking number $lk(K',K)$ is well-defined and equals to an integer $n$, where $\ker\big{(}H_1(M\setminus K)\rightarrow H_1(M)\big{)}\cong\Z$ is generated by a meridional disc with homology class $[\mu]$ and $[K']\in H_1(M\setminus K)$ is in this kernel with $[K']=n[\mu]$. The value of $tb_\Sigma(K)$ depends only on the relative homology class $[\Sigma]\in H_2(M,K;\Z)$ since the Seifert framing for null-homologous $K$ is well-defined, so $tb_{\Sigma'}(K)=tb_\Sigma(K)$ for any $\Sigma'$ with $[\Sigma']=[\Sigma]$, and for a choice of another $[\Sigma'']\in H_2(M,K;\Z)$, $tb_{\Sigma''}(K)-tb_\Sigma(K)=tw_K(\xi,Fr_{\Sigma''})-tw_K(\xi,Fr_\Sigma)=tw_K(Fr_{\Sigma''},Fr_\Sigma)$, which could be nontrivial if $H_2(M;\Z)\neq 0$.

Additionally, for a null-homologous Legendrian knot $K$ in $(M,\xi)$, the {\em rotation number} $r_\Sigma(K,\sigma)$ is defined for a trivialization of $\xi\rvert_\Sigma$, given by a nonzero section $\sigma:\xi\rvert_\Sigma\rightarrow \Sigma\times\R^2$, whose restriction to $K$ yields a non-zero vector field $v_K$ tangent to $K$. The rotation number $r_\Sigma(K,\sigma)$ of $K$ with respect to this trivialization is given by $e(\xi,v_K)([\Sigma])$, where $e(\xi)\in H^2(M;\Z)$ is the Euler class of $\xi$ (for a choice of coorientation of $\xi$) and $e(\xi,v_K)$ is the relative Euler class of $\xi\rvert_K$.  

Equivalently, for a non-zero vector field tangent to $K$, the rotation number is the obstruction of extending $v_K$ to a nonzero vector field in $\xi\rvert_\Sigma$. More directly, one can compute the rotation number as the winding number $w_\sigma(v_K)$ of the image of $v_K$ in $\R^2$ under the map $\xi\rvert_K\rightarrow K\times\R^2$ which is the restriction of the trivialization $\xi\rvert_\Sigma$ to $K$. The rotation number does not depend on the trivialization $\sigma$ (so we will omit it from the notation), but it does depend on the relative homology class $[\Sigma]\in H_2(M,K;\Z)$ and on the orientation of $\Sigma$. For a choice of another relative homology class $[\Sigma']\in H_2(M,K;\Z)$, we have $r_{\Sigma'}(K)-r_\Sigma(K)=\pm e(\xi)([\Sigma']-[\Sigma])$, where $[\Sigma']-[\Sigma]\in H_2(M;\Z)$. 

For a null-homologous transverse knot $K\subset (M,\xi)$ in an oriented contact 3-manifold $M$, let $\Sigma\subset M$ be a Seifert surface for $K$, then we can define the {\em self-linking number} of $K$ as $sl_\Sigma (K)=K'\cdot \Sigma=lk_\Sigma(K',K)$, where $K'$ is a push-off of $K$ in the direction of (the restriction along $K$ of) a nonzero vector field $v$ on $\Sigma$ (such a vector field exists since $\xi\rvert_\Sigma$ is trivial). 

The self-linking number can be interpreted as the obstruction of extending a non-zero vector field $v_K$ (which points out of $\Sigma$) along $K$ in $\xi\cap T\Sigma$ to a non-zero vector field in $\xi\rvert_\Sigma$, in which case a push-off $K'$ of $K$ along $v_K$ would have linking $0$ with $K$ relative to $\Sigma$. Similarly to the rotation number, upon fixing a coorientation of the contact plane field, a choice of a different relative homology class $[\Sigma']\in H_2(M,K;\Z)$ gives us $sl_{\Sigma'}(K)-sl_\Sigma(K)=\pm e(\xi)([\Sigma']-[\Sigma])$, where the sign depends on whether the orientation of $K$ coincides with the transverse orientation (or coorientation) of $\xi$. For more details, read \cite{aebisher, etnyre:intro, etnyre:knots, geiges}. 

\section{Basic Definitions and Properties}

We are considering homologous oriented knots $K$ and $J$ in a contact 3-manifold $M$. Equivalently, the oriented link $K\cup J$ is null-homologous in $M$, so the usual argument generalizes directly to show there exists a smooth embedded surface $\Sigma\subset M$, where $[\Sigma]\in H_2(M,K\cup J;\Z)$ is mapped to $[K\cup J]\in H_1(K\cup J;\Z)$ by the boundary homomorphism, that is, $\Sigma$ is a Seifert surface for $K\cup J$.

\begin{remark}
Always orient $K$ and $J$ as boundary components of the oriented $\Sigma$ so that $[\partial\Sigma]=[K]-[J]$.
\end{remark}

We generalize the classical invariants to the case when $K$ is homologous to another knot $J$ in $(M,\xi)$ as follows.

\begin{definition}
Let $K$ and $J$ be homologous Legendrian knots in a contact 3-manifold $(M,\xi)$ oriented accordingly with $K\cup J=\partial\Sigma$ for an oriented embedded Seifert surface $\Sigma$ so that $[\partial\Sigma]=[K]-[J]$. Define the {\em Thurston-Bennequin invariant of $K$ relative to $J$} by $$\widetilde{tb}_\Sigma(K,J):=tw_K(\xi,Fr_{\Sigma})-tw_{J}(\xi,Fr_{\Sigma}),$$ where $Fr_\Sigma$ denotes the Seifert framing that $K$ (resp. $J$) inherits from $\Sigma$, and $tw(\xi,Fr_\Sigma)$ denotes the number of $2\pi$-twists (with sign) of the contact framing relative to $Fr_\Sigma$ along $K$ or $J$. For push-offs $K'$ and $J'$ of $K$ and $J$ in the direction normal to the contact planes, $\widetilde{tb}_\Sigma(K,J)=K'\cdot\Sigma-J'\cdot\Sigma=lk_\Sigma(K',K)-lk_\Sigma(J',J)$.
\end{definition}

\begin{definition} 
Let $K$ and $J$ be homologous Legendrian knots in a contact 3-manifold $(M,\xi)$ oriented accordingly with $K\cup J=\partial\Sigma$ for an oriented embedded Seifert surface $\Sigma$ so that $[\partial\Sigma]=[K]-[J]$. The restriction to $K$ of the trivialized contact 2-plane field $\xi\rvert_\Sigma$ gives a map $\sigma:\xi\rvert_K\rightarrow K\times\R^2$, under which a non-zero tangent vector field $v_K$ to $K$ traces out a path of vectors in $\R^2$. We can then compute the winding number $w_\sigma(v_K)$ and similarly for $J$. Then define the \textit{relative rotation number of $K$} by $$\widetilde{r}_\Sigma(K,J):=w_\sigma(v_K)-w_\sigma(v_J).$$ Equivalently, $\widetilde{r}_\Sigma(K,J)=e(\xi,v_K\cup v_J)([\Sigma])$.
\end{definition}

\begin{definition} 
Let $K$ and $J$ be homologous transverse knots in a contact 3-manifold $(M,\xi)$ oriented accordingly with $K\cup J=\partial\Sigma$ for an oriented embedded Seifert surface $\Sigma$ so that $[\partial\Sigma]=[K]-[J]$. The contact 2-plane field $\xi$ is trivial over $\Sigma$, so there exists a nonzero vector field $v$ in $\xi\rvert_\Sigma$. Take $K'$ and $J'$ to be the push offs of $K$ and $J$ in the direction of $v$. Then define the {\em relative self-linking number} of $K$ with respect to $J$ by 
$$\widetilde{sl}_\Sigma (K,J):=K'\cdot\Sigma-J'\cdot\Sigma.$$
\end{definition}

Next we give some basic properties of the relative invariants and discuss their dependence on the Seifert surface.

Let $K$, $K'$, and $J$ be oriented Legendrian knots in a contact 3-manifold accordingly with $K\cup J=\partial\Sigma$ so that $[\partial\Sigma]=[K]-[J]$ for an oriented embedded surface $\Sigma$, and let $K'$ be homologous to $K$ and $J$ via a surface $\Sigma'$ oriented accordingly with $-K\cup K'=\partial\Sigma'$ so that $[\partial\Sigma']=[K]-[K']$ and assume $\Sigma''$ is and embedded surface oriented with $-K'\cup -J=\partial\Sigma''$ so that $[\partial\Sigma'']=[-K']-[-J]$. Assume that $\Sigma$, $\Sigma'$ and $\Sigma''$ are disjoint away from their common boundaries. Then 
$$\widetilde{tb}_\Sigma(K,J)=\widetilde{tb}_{\Sigma'}(-K,K')+\widetilde{tb}_{\Sigma''}(-K',-J).$$ 
Note also that $\widetilde{tb}_\Sigma(K,J)=-\widetilde{tb}_\Sigma(J,K)$, and $\widetilde{tb}_\Sigma(K,J)=\widetilde{tb}_{-\Sigma}(-K,-J)$. To see this, check that $tw_K(\xi,Fr_\Sigma)$ and $tw_J(\xi,Fr_\Sigma)$ are independent of the orientations of $K$ and $J$. Under (de)stabilization (see \cite{etnyre:knots, geiges}), we have $tb_{\Sigma'}(S_\pm(K))=tb_\Sigma(K)-1$, therefore $\widetilde{tb}_{\Sigma'}(S_\pm(K))=\widetilde{tb}_\Sigma(K)-1$ and $\widetilde{tb}_{\Sigma'}(S_\pm(K), S_\pm(J))=\widetilde{tb}_\Sigma(K,J)$, where $\Sigma'$ denotes the surface after (de)stabilization.

\begin{lemma}\label{framing} 
Consider a contact 3-manifold $(M,\xi)$ and homologous Legendrian knots $K$ and $J$ in $M$ oriented accordingly with $K\cup J=\partial\Sigma$ so that $[\partial\Sigma]=[K]-[J]$, where $\Sigma\subset M$ is an oriented embedded surface. Then the Seifert framing $Fr_\Sigma$ on $K$ (resp., on $J$) is uniquely defined by the relative homology class $[\Sigma]\in H_2(M,K\cup J;\Z)$.
\end{lemma}

\begin{proof} 
Take neighborhoods $N_K$ and $N_J$, we have the maps (with $\Z$-coefficients) $$H_2(M, K\cup J)\xrightarrow{i} H_2(M,N_K\cup N_J)$$ $$H_2(M,N_K\cup N_J)\xrightarrow{f}H_2(M\setminus (N_K\cup N_J), \partial N_K\cup \partial N_J)\xrightarrow{\partial}H_1(\partial N_K\cup \partial N_J).$$ The map $i$ is an isomorphism because $M\setminus (K\cup J)$ deformation retracts onto $M\setminus (N_K\cup N_J)$. The map $f$ is an isomorphism from excising $N_K\cup N_J$, and $\partial$ is the usual boundary map on relative homology. Let $\Psi=i\circ f\circ \partial$, then $\Psi([\Sigma])$ gives the homology classes $[K']\in H_1(\partial N_K)$ and $[J']\in H_1(\partial N_J)$ of corresponding push-offs of $K$ and $J$ into $\Sigma$. So $[K']$ and $[J']$ are uniquely defined by $[\Sigma]\in H_2(M, K\cup J)$. 
\end{proof}

The following states that framings on $K$ and $J$ induced by different relative homology classes in $H_2(M,K\cup J)$ differ by the same number of twists along each boundary component.

\begin{lemma} \label{independent}
Consider knots $K$ and $J$ oriented accordingly with $\partial\Sigma_1=K\cup J$ for an oriented Seifert surface $\Sigma_1$ in a 3-manifold $M$ so that $[\partial\Sigma_1]=[K]-[J]$. Then for any other oriented Seifert surface $\Sigma_2$ with $K\cup J=\partial\Sigma_2$ so that $[\partial\Sigma_2]=[K]-[J]$,
$$tw_K(Fr_{\Sigma_1},Fr_{\Sigma_2})=tw_J(Fr_{\Sigma_1},Fr_{\Sigma_2}).$$
\end{lemma}

\begin{proof} Consider a homology class $[\Sigma_1]\in H_2(M,J\cup K)$ and a surface $\Sigma_2$ with $[\Sigma_1]=[\Sigma_2]\in H_2(M, J\cup K)$, then $[\Sigma_1]-[\Sigma_2]$ is a class in $H_2(M)$. So any other surface for $K\cup J$ is obtained from $\Sigma_1$ by adding on a closed surface $A$ of some homology class $[A]\in H_2(M)$. We construct a surface $\Sigma_1'$ with boundary $J\cup K$ representing the class $[\Sigma_1]+[A]$ and show that the above equality is true for $\Sigma_1'$. 

First, observe that $J\cdot A'=K\cdot A'$ for any $A'\in[A]$. We can ensure this is true by modifying the interior of $A'$ as follows. Replace small closed 2-discs in $A'$ around consecutive intersection points (with opposite signs) of $A'$ along $J$ with a cylinder whose boundary components are the boundaries of the two 2-discs and which runs in a neighborhood of $J$. This eliminates such intersection points in pairs. 

The intersections between $\Sigma$ and $A'$ consist of $\rvert J\cdot A'\rvert=\rvert J\cdot A'\rvert$ ribbon arcs and possibly some circles on the interior of $\Sigma$. We first eliminate the circle intersections by cutting-open the oriented surfaces along each circle and gluing up the pieces to obtain a new oriented surface (there is an unique obvious way to do this, which depends on the orientations of the surfaces, see Figure \ref{f10} below, in particular, reversing the orientation of one of the surfaces results in the alternate gluing up of the cut open surfaces). 

\begin{figure}[!ht]
\includegraphics[scale=.45]{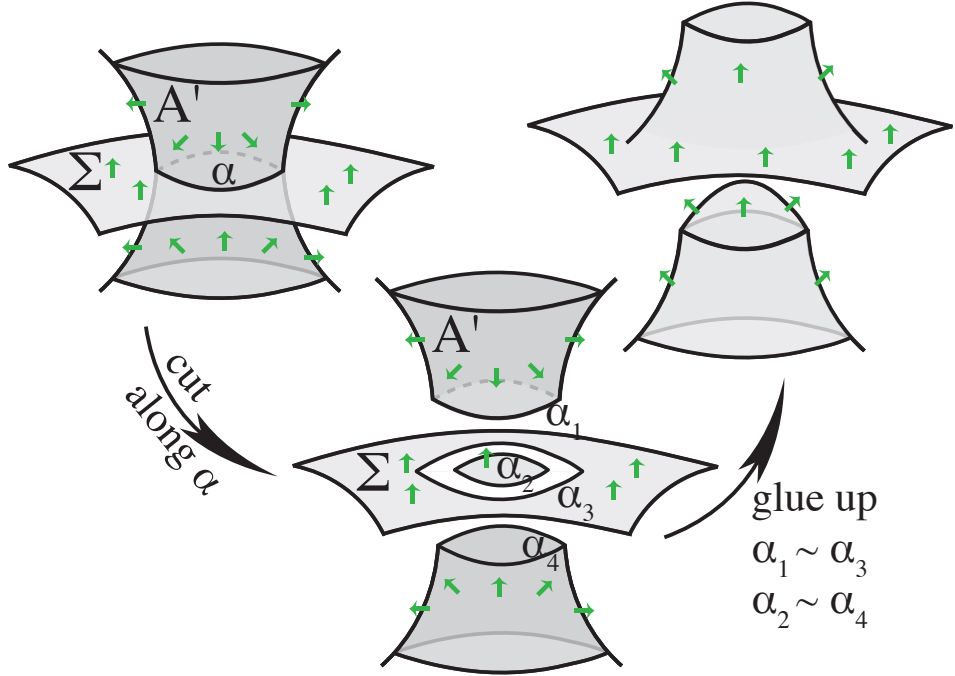}
\caption{Smoothing circle intersections between $\Sigma$ and $A'$.}\label{f10}
\end{figure}

Next, consider the ribbon intersection arcs and label them $\delta_i,i=1,\dots,\rvert J\cdot A'\rvert$. Note that each arc $\delta_i$ runs on $\Sigma$ from a point $a_i\in K$ to a point $b_i\in J$ with $\partial\delta_i=\{a_i,b_i\}$, so that on $A'$, $a_i$ and $b_i$ are interior points and $\delta_i$ is an interior arc (see Figure \ref{f1}). 

To resolve the $\delta_i$, cut along the interior of each and glue the surfaces (only one way to do this so that we get an oriented surface, see Figure \ref{f11} below) and note that the twists at each endpoint of the arc are opposite.

\end{proof}

Lemma \ref{independent} implies that the relative Thurston-Bennequin invariant is independent not only of the Seifert surface but also of the relative homology class. This will be useful when proving invariance under Legendrian isotopies.

For the relative rotation number, consider Legendrian knots $K$ and $J$ oriented as boundary components of an oriented embedded surface $\Sigma$, we have $\widetilde{r}_\Sigma(K,J)$, $\widetilde{r}_{-\Sigma}(-K,-J)=-\widetilde{r}_\Sigma(K,J)$. Under (de)stabilization of $K$, only the winding number along $K$ would change by $\pm 1$, therefore, $\widetilde{r}_{\Sigma'}(S_\pm(K),J)=\widetilde{r}_\Sigma(K,J)\pm1$, where $\Sigma'$ denotes the new surface after (de)stabilization. The relative rotation number does not depend on the trivialization of $\xi\rvert_\Sigma$ but it does depend on the Seifert surface through its relative homology class $[\Sigma]\in H_2(M,K\cup J;\Z)$ and its orientation.

\begin{figure}[!ht]
\includegraphics[scale=.52]{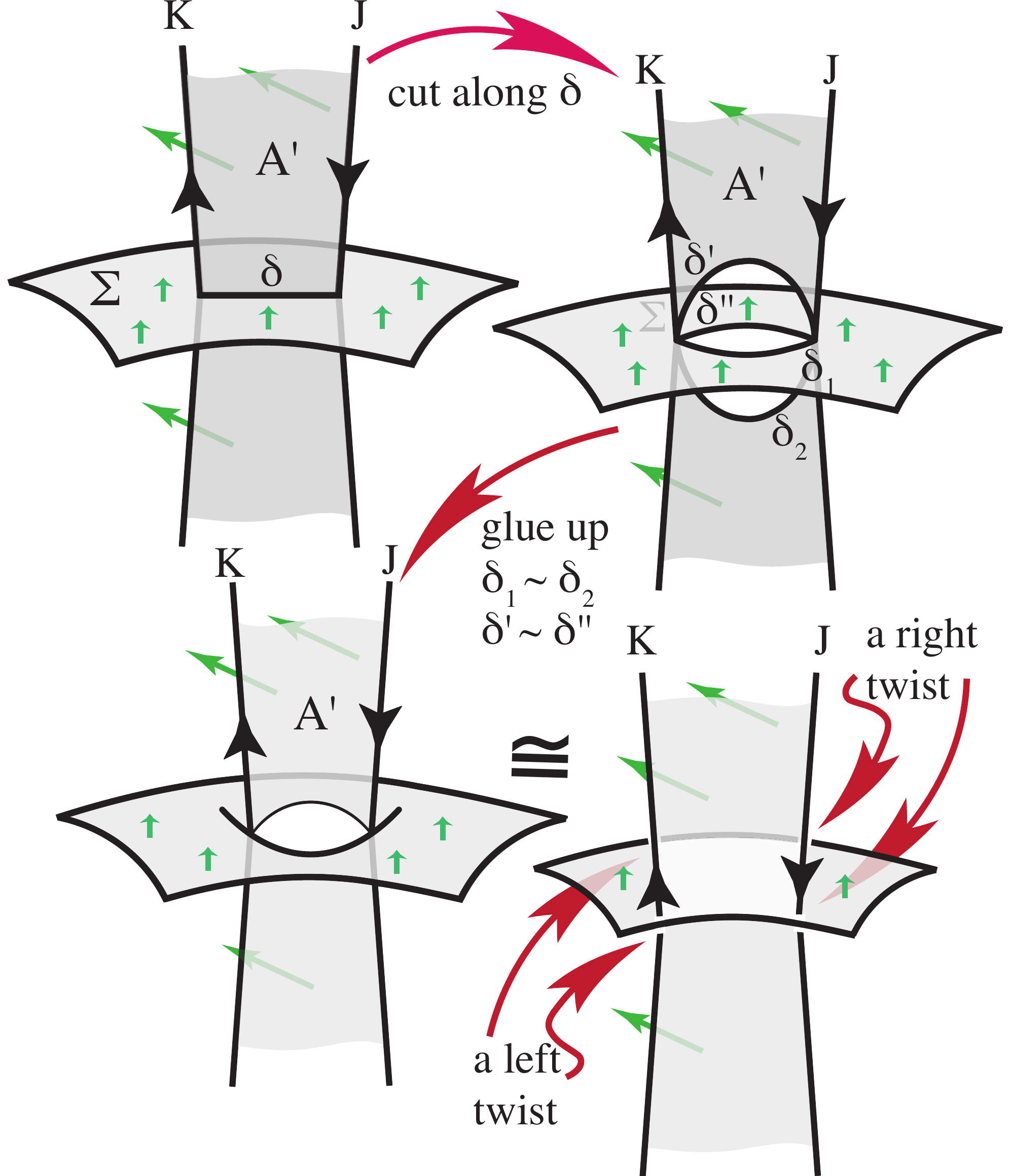}
\caption{Smoothing ribbon intersections between $\Sigma$ and $A'$ adds ONE twist along each $J$ and $K$, with opposite signs.}\label{f11}
\end{figure}

\begin{lemma}\label{trivializationdependence} 
For homologous Legendrian knots $K$ and $J$ in a contact 3-manifold $(M,\xi)$ oriented accordingly with $K\cup J=\partial\Sigma$ for an embedded oriented surface $\Sigma\subset M$ so that $[\partial\Sigma]=[K]-[J]$, the relative rotation number $\widetilde{r}_\Sigma(K,J)$ does not depend on the choice of trivialization of $\xi\rvert_\Sigma$. 
\end{lemma}

\begin{proof} 
Consider two trivializations $\sigma_i:\xi\rvert_\Sigma\rightarrow\Sigma\times\R^2$ given by two nontrivial sections $Y_1$ and $Y_2$ so that there exists a function $f:K\rightarrow SO(2)\cong S^1$ such that $Y_1(p)=f(p)Y_2(p)$ for all $p\in K\cup J$. Therefore, the winding number $\omega_{\sigma_1}(Y_1\rvert_K)=\omega_{\sigma_2}(Y_2\rvert_K)-deg(f,K)$, where the degree of $f$ over $K$ can be computed by the induced homomorphism $\Z\cong H_1(K)\rightarrow H_1(K)\cong\Z$ in homology. Similarly, $\omega_{\sigma_1}(Y_1\rvert_J)=\omega_{\sigma_2}(Y_2\rvert_J)-deg(f,J)$. In particular, since $H_1(J)\cong H_1(K)\cong H_1(S^1)\cong\Z$, we have that $deg(f,K)=deg(f,J)$, so 
$$\widetilde{r}_\Sigma (K,J,\sigma_1)=\omega_{\sigma_1}(Y_1\rvert_K)-\omega_{\sigma_1}(Y_1\rvert_J)=\omega_{\sigma_2}(Y_2\rvert_K)-deg(f,K)-\omega_{\sigma_2}(Y_2\rvert_J)+deg(f,J)$$
$$=\omega_{\sigma_2}(Y_2\rvert_K)-\omega_{\sigma_2}(Y_2\rvert_J)=\widetilde{r}_\Sigma (K,J,\sigma_2).$$ \end{proof}

\begin{lemma}\label{independent2} 
Let $K$ and $J$ be homologous Legendrian knots oriented as $K\cup J=\partial\Sigma_1=\partial\Sigma_2$ for oriented Seifert surfaces  in a contact 3-manifold $M$ with relative homology classes $[\Sigma_1],[\Sigma_2]\in H_2(M,K\cup J;\Z)$ so that $[\partial\Sigma_i]=[K]-[J]$. Then
$$\widetilde{r}_{\Sigma_1}(K,J)-\widetilde{r}_{\Sigma_2}(K,J)=e(\xi)([\Sigma_1]-[\Sigma_2].$$ 
\end{lemma}

Notice that if $[\Sigma_1]=[\Sigma_2]$, the lemma shows that the relative rotation number is independent of the particular relative homology representative.

\begin{proof} This is a direct generalization of the classical argument (see \cite{geiges}). Consider the closed oriented surface $\Sigma_1\cup(-\Sigma_2)$, the union taken along $K\cup J$, which is oriented as the boundary of $\Sigma$. The two trivializations $\xi\rvert_{\Sigma_1}$ and $\xi\rvert_{\Sigma_2}$ coincide along a segment $K_1$ on $K$ and a segment $J_1$ on $J$, which means that we have a trivialization of $\xi$ over $\Sigma_1\cup-\Sigma_2$ with a 2-disc $D_K$ removed and a 2-disc $D_J$ removed, where $D_K\cap K=K\setminus K_1$ and $D_J\cap J=J\setminus J_1$. Since $K$ and $J$ are oriented as boundary components of $\Sigma_1$, there are compatible orientations for $D_K$ and $D_J$ which coincide with their natural orientation from $\Sigma_1\cup-\Sigma_2$. Over $D_K$ and $D_J$, there exist unique trivializations of $\xi$ up to homotopy. Assume that the trivialization $\xi\rvert_{D_K}$ coincides along $D_K\cap K$ with the trivialization defined by the positive tangent vector $v_K$ to $K$, and similarly, that the trivialization $\xi\rvert_{D_J}$ and the trivialization of $\xi$ over $D_J$ given by the positive tangent vector $v_J$ along $J$. Then by definition we have $\widetilde{r}_{\Sigma_1}(K,J)=w_{\sigma_1}(v_K)-w_{\sigma_1}(v_J)$ and $\widetilde{r}_{\Sigma_2}(K,J)=w_{\sigma_2}(v_K)-w_{\sigma_2}(v_J)$, where $\sigma_i$ denotes the trivialization $\xi\rvert_{\Sigma_i}$. Therefore, we have
$$\widetilde{r}_{\Sigma_1}(K,J)-\widetilde{r}_{\Sigma_2}(K,J)=w_{\sigma_1}(v_K)-w_{\sigma_1}(v_K)-w_{\sigma_2}(v_K)+w_{\sigma_2}(v_K).$$
Since $-\widetilde{r}_{\Sigma_2}(K,J)=\widetilde{r}_{-\Sigma_2}(-K,-J)$, the above expression becomes
$$w_{\sigma_1}(v_K)+w_{\sigma_2}(-v_K)-\Big{(}w_{\sigma_1}(v_J)+w_{\sigma_2}(-v_J)\Big{)}.$$

First, consider $w_{\sigma_1}(v_K)+w_{\sigma_2}(-v_K)$. The first winding number is equal to the number of rotations of $v_K$ relative to $\sigma_1$ along $K\cap D_K$, which is oriented homotopic rel endpoints to the segment $\partial D_K\cap\Sigma_1$. The second winding number represents the number of rotations of $-v_K$ along $-K\cap D_K$, which is homotopic rel boundary to the segment $\partial D_K\cap-\Sigma_2$.

We have analogous interpretation of the term $w_{\sigma_1}(v_J)+w_{\sigma_2}(-v_J)$.

Next take a segment $\alpha\subset\Sigma_1\cup (-\Sigma_2)$ from a point on $\partial D_K$ to a point on $\partial D_J$ and consider a neighborhood $N(\alpha)$ of $\alpha$, whose boundary consists  of two arcs parallel to $\alpha$ and of an arc $\alpha_K$ along $\partial D_K$ such that $\xi\rvert_{\alpha_K}$ rotates once and, similarly, an arc $\alpha_J$ along $\partial D_J$ such that $\xi\rvert_{\alpha_J}$ also rotates once. Since $\xi\rvert_{N(\alpha)}$ is trivial, $\xi$ rotates in the same way along $\alpha_K$ and $\alpha_J$, however, the orientations of $\alpha_K$ and $\alpha_J$ are reversed by an orientation-preserving isotopy between them through the product structure $\alpha\times[0,1]\cong N(\alpha)$. Therefore, $N(\alpha)$ does not modify the value of the quantity 
$$w_{\sigma_1}(v_K)+w_{\sigma_2}(-v_K)-\Big{(}w_{\sigma_1}(v_J)+w_{\sigma_2}(-v_J)\Big{)}.$$

Then the Euler number $e(\xi)([\Sigma_1]-[\Sigma_2]$ measures the number of rotations of the trivialization of $\xi$ over $(\Sigma_1\cup -\Sigma_2)\setminus (D_K\cup N(\alpha)\cup D_J)$ along the boundary of the 2-disc $D_K\cup N(\alpha)\cup D_J$ relative to a constant vector field which extends $v_K$ and $-v_J$ over $N(\alpha)$. This is equal to the number of rotations of the trivialization of $\xi$ over $(\Sigma_1\setminus D_K)\cup (-\Sigma_2\setminus D_J)$ over $\partial D_K\cup D_J$ with respect to extensions of $v_K$ and $-v_J$, respectively. Therefore, $\widetilde{r}_{\Sigma_1}(K,J)-\widetilde{r}_{\Sigma_2}(K,J)=e(\xi)([\Sigma_1]-[\Sigma_2])$.

A direct way of obtaining this result is to choose $\sigma_i$ so that they are both zero along $J$ so the argument directly reduces to the case when we compare the rotations of the trivializations only along $K$, which is just the classical case.
\end{proof}

For the relative self-linking number, consider the transverse knots $K$ and $J$ oriented as boundary components of the oriented surface $\Sigma$, over which $\xi$ is trivial and, in particular, it is cooriented along $\partial\Sigma$. Then $\widetilde{sl}_\Sigma(K,J)$ depends on how the orientations of the components of $\partial\Sigma$ compare with orientation of (the normal direction of) the contact planes, so reversing the orientation of either $\Sigma$ or $\xi\rvert_\Sigma$ causes $\widetilde{sl}_\Sigma(K,J)$ to reverse sign. 

\begin{lemma}\label{independent3}
Consider a contact 3-manifold $(M,\xi)$ and homologous transverse knots $K$ and $J$ in $M$ oriented as $K\cup J=\partial\Sigma$ for an embedded oriented surface $\Sigma\subset M$ so that $[\partial\Sigma]=[K]-[J]$. The relative self-linking number $\widetilde{sl}_\Sigma(K,J)$ does not depend on the choice of trivialization of $\xi\rvert_\Sigma$. 
\end{lemma}

\begin{proof} 
Follows by the same argument as in the proof of Lemma \ref{trivializationdependence}.
\end{proof}

\begin{lemma} 
Let $\Sigma_1$ and $\Sigma_2$ be Seifert surfaces for the homologous oriented transverse knots $K$ and $J$ in a contact 3-manifold $M$ representing the relative homology classes $[\Sigma_1],[\Sigma_2]\in H_2(M,K\cup J;\Z)$ so that $[\partial\Sigma_i]=[K]-[J]$. Then
$$\widetilde{sl}_{\Sigma_1}(K,J)-\widetilde{sl}_{\Sigma_2}(K,J)=e(\xi)([\Sigma_1]-[\Sigma_2].$$ 
\end{lemma}

Notice that if $[\Sigma_1]=[\Sigma_2]$, the lemma shows that the relative self-linking number is independent of the particular relative homology representative. 

\begin{proof} 
The argument is similar to the proof of Lemma \ref{independent}. Assume without loss of generality that the coorientation of $\xi$ makes $K$ positively transverse. Again take the union $\Sigma_1\cup (-\Sigma_2)$ and arrange that the trivializations of $\xi\rvert_{\Sigma_1}$ and $\xi\rvert_{\Sigma_2}$ coincide along an interval $K_1=K\setminus D_K$ and an interval $J_1=J\setminus D_J$. Thus we obtain a trivialization of $\xi$ over $(\Sigma_1\cup -\Sigma_2)\setminus (D_K\cup D_J)$ given by a nonzero vector field $X$. Note that $\xi\rvert_{D_K\cup D_J}$ is also trivial, the trivialization of $\xi\rvert_{D_K}$ given by a constant vector field $v_K$ tangent to $K_1=K\cap D_K$ and the trivialization of $\xi\rvert_{D_J}$ given by a constant vector field $v_J$ tangent to $-J_1=-J\cap D_J$. Then, as in Lemma \ref{independent}, one can use a neighborhood $N(\alpha)$ of an arc between $\partial D_K$ and $D_J$ to reduce the computation of the trivialization of $\xi$ to the boundary of the 2-disc $D_K\cup N(\alpha)\cup D_J$ and see that it computes an Euler number. This Euler number $e(\xi)([\Sigma_1]-[\Sigma_2])$ measures the difference of the winding number of $X_K$ with respect to $X$ along $\partial D_K$ minus the winding number of $X_J$ with respect to $X$ along $\partial D_J$ (note that $\partial D_K$ and $\partial D_J$ are oriented as boundaries of $D_K$ and $D_J$ on the oriented $\Sigma_1\cup -\Sigma_2$, so the negative sign comes form the fact that $v_J$ is pointing in the direction of $-J$). 

Now let $K'$ and $J'$ be push-offs of $K$ and $J$, respectively, in the direction of $X$. Then by definition,
$$\widetilde{sl}_{\Sigma_1}(K,J)=lk_{\Sigma_1}(K',K)-lk_{\Sigma_1}(J',J)=K'\cdot\Sigma_1-J'\cdot\Sigma_1$$ and $$\widetilde{sl}_{\Sigma_2}(K,J)=lk_{\Sigma_2}(K',K)-lk_{\Sigma_2}(J',J)=K'\cdot\Sigma_2-J'\cdot\Sigma_2.$$ From the construction, there is no contribution to $K'\cdot\Sigma_i$ along $K_1=K\setminus D_K$ and there is no contribution to $J'\cdot\Sigma_i$ along $J_1=J\setminus D_J$ for $i=1,2$. Then along $K\cap D_K$, $K'\cdot\Sigma_i$ is equal to the winding number of $X$ relative to $X_K$ measured along this segment, because $K_1$ is homotopic, rel bondary, to $\partial D_K\cap\Sigma_1$, by an orientation-reversing homotopy, and $K_1$ is homotopic, rel bondary, to $\partial D_K\cap-\Sigma_2$ by an orientation-preserving homotopy. Therefore, this computes $$K'\cdot\Sigma_1-K'\cdot\Sigma_2.$$ Similarly, to measure the winding number of $X$ relative to $X_J$, we look at $J\cap D_J$, keeping in mind that $-J\cap D_J$ is homotopic to $\partial D_J\cap\Sigma_1$ by an orientation-preserving homotopy and $-J\cap D_J$ is homotopic to $\partial D_J\cap-\Sigma_2$ by an orientation-reversing homotopy. We are working with $-\Sigma_2$, so switching the signs, we obtain $$-J'\cdot\Sigma_1+J'\cdot\Sigma_2,$$ which computes precisely the last two remaining terms of the difference of the relative self-linking numbers coming from $\Sigma_1$ and $\Sigma_2$, completing the proof.
\end{proof}

\section{Seifert Surfaces and Smooth Isotopies}

In this section, we study the problem of finding a Seifert surface for oriented $K\cup J$ under smooth isotopies. Given a surface $\Sigma\subset (M,\xi)$ with $\partial\Sigma=K\cup J$, consider a smooth isotopy $\varphi_t(K), t\in[0,1],$ of $K$ which fixes $J$. If the isotopy lives in the complement of $J$, then it extends to a global isotopy by the Isotopy Extension Theorem, and, in particular, to an isotopy of $\Sigma$ through embedded surfaces. More formally, $\varphi_t(\Sigma)$ is a Seifert surface for $\varphi_t(K)\cup\varphi_t(J)$ for all $t\in[0,1]$. Note that there are versions of the Isotopy Extension Theorem for transverse and Legendrian isotopies. If $\varphi_t(K)$ intersects $\varphi_t(J)=J$ for some $t$, we show how to construct a surface for $\varphi_t(K)\cup\varphi_t(J)$ after the intersection occurs.

\begin{lemma}\label{construction} 
Consider homologous knots $K$ and $J$ in a contact 3-manifold $(M,\xi)$ oriented as boundary components of an embedded oriented surface $\Sigma\subset M$ so that $[\partial\Sigma]=[K]-[J]$, and consider a smooth isotopy $\varphi_t$ of $K$ which fixes $J$. Additionally, assume that $\varphi_t(K)$ intersects $J$ transversely in a point $p\in J$, that is, $\varphi_s(K)\pitchfork J=\{p\}$ and $\varphi_t(K)\cap J=\emptyset$ for all $t\neq s$. For small $\epsilon>0$, let $K_\pm=\varphi_{s\pm\epsilon}(K)$ and let $\Sigma_-=\varphi_{s-\epsilon}(\Sigma)$, then there exists an oriented embedded surface $\Sigma_+\subset M$ with $\partial\Sigma_+=K_+\cup J$.
\end{lemma}

\begin{proof}
Since $K_-$ and $K_+$ are isotopic, they are homologous, so there exists an embedded oriented surface $A$ with $\partial A=-K_-\cup K_+$, and we can arrange that $A\pitchfork J=\{p\}$. We construct $\Sigma_+$ out of the surfaces $A$ and $\Sigma_-$ by eliminating their intersections. Intersection arcs that run boundary-to-boundary on one surface and intersect only the interior of the other surface are called \textit{ribbon arcs} (Figure \ref{f1}).

Ribbon arcs can be eliminated by locally pushing the interior of the one surface (e.g., $\Sigma_2$ in Figure \ref{f4}) which contains the arc in its interior across the part of the other surface (e.g., $\Sigma_1$ in Figure \ref{f4}) that the arc bounds with the boundary (this requires an innermost-arc-first procedure).

\begin{figure}[!ht]
\centering
\includegraphics[scale=.3]{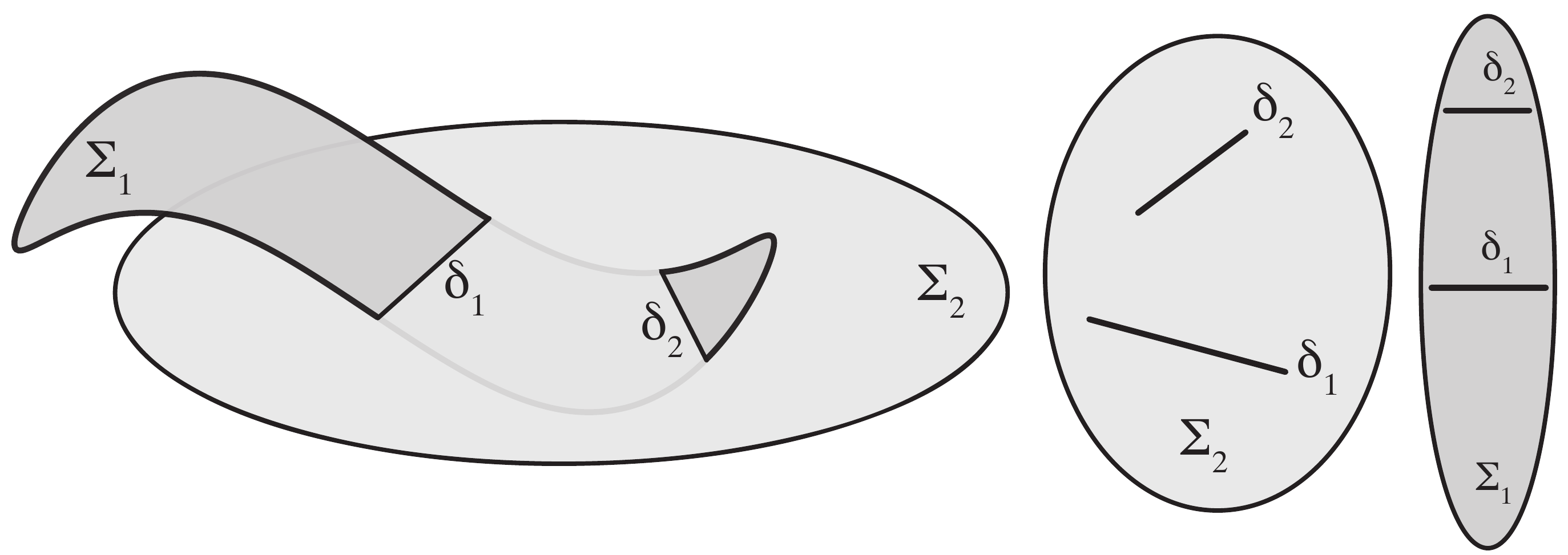}
\caption{Ribbon intersections $\delta_1$ and $\delta_2$.}\label{f1}
\end{figure} 

If this part of surface $\Sigma_2$ contains any other intersections with $\Sigma_1$ (circles), eliminate them first by locally isotoping the interior of $\Sigma_2$ across (see Figure \ref{f2} for an example involving $\Sigma_-$ and $A$). 

\begin{figure}[!ht]
\centering
\includegraphics[scale=.3]{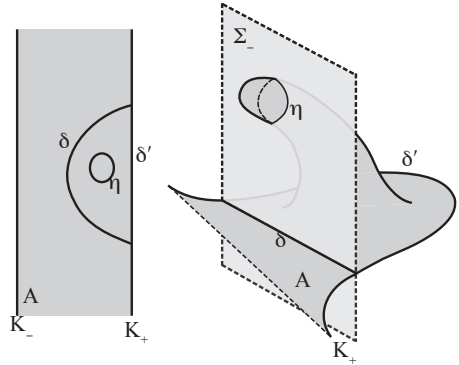}
\caption{Resolving an innermost trivial ribbon intersection $\delta$.}\label{f2}
\end{figure}

Note that for this to be always possible, we need $M$ to be irreducible. 

\begin{figure}[!ht]
\centering
\includegraphics[scale=.25]{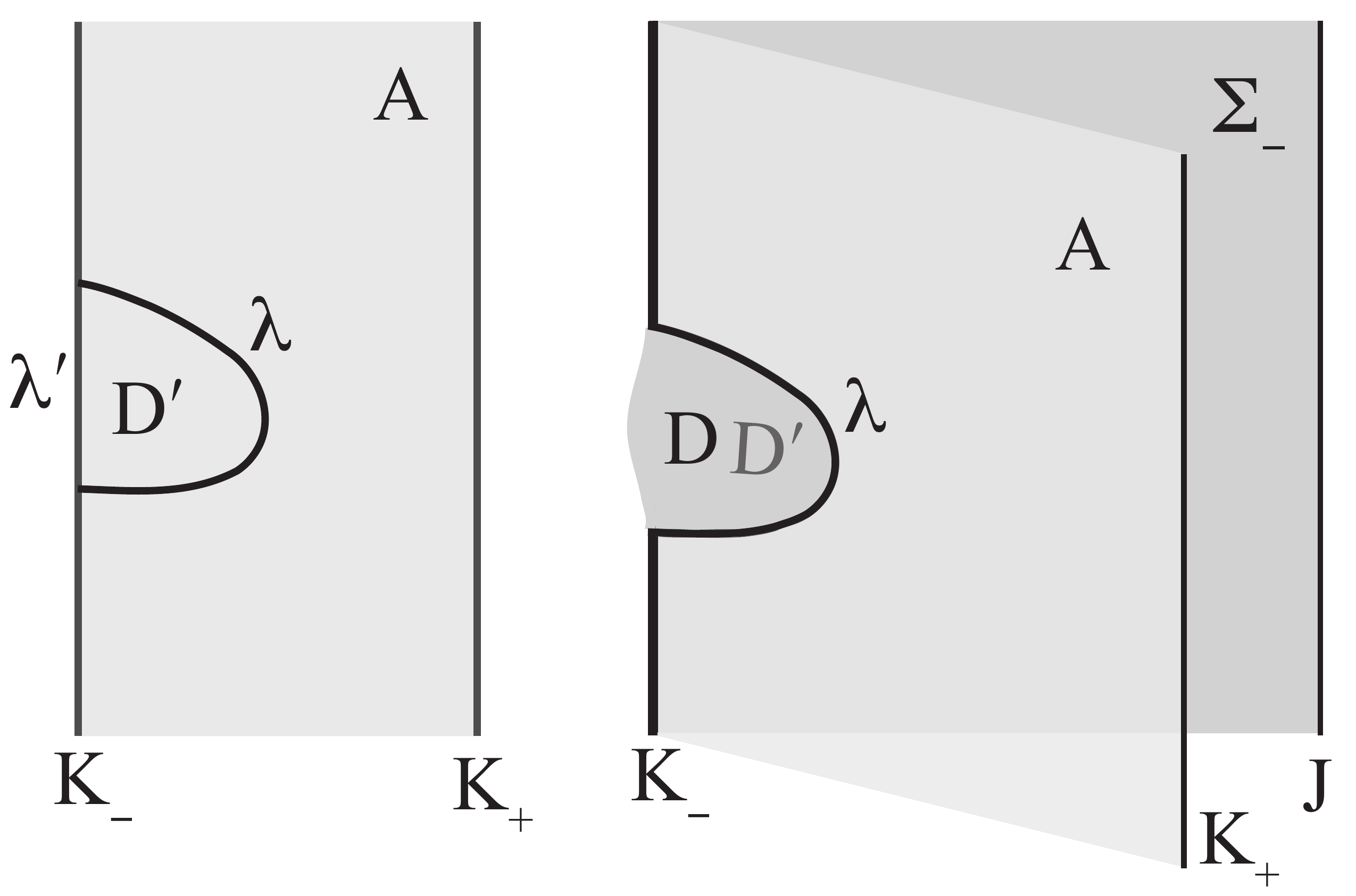}
\caption{Resolving an intersection arc $\lambda$ parallel to $K_-$.}\label{f3}
\end{figure}

However, $A$ lives in a solid torus (which is irreducible) neighborhood and all the isotopies of the interior of $\Sigma_-$ can be restricted to live in this solid torus, so no such general assumptions for $M$ are necessary.

So we assume there are no ribbon intersections and next eliminate non-ribon intersection arcs (parallel to $K_-$) by isotoping the interior of $\Sigma_-$ locally across the half-disc that such arc bounds with $K_-$ (see Figure \ref{f3}).

Next we eliminate circle intersections between $\Sigma_-$ and $A$, which we claim can not be  multiples of $[K_+]\in H_1(A;\Z)$. Assume that $\eta\subset A$ is a multiple of $[K_-]$. Recall that $J\cap A=\{p\}$, which produces the intersection arc $\alpha$ that runs from $J$ to $K_-$ on $\Sigma_-$ and $p$ to the boundary $K_-$ on $A$ (Figure \ref{f4}). 

\begin{figure}[!ht]
\centering
\includegraphics[scale=.35]{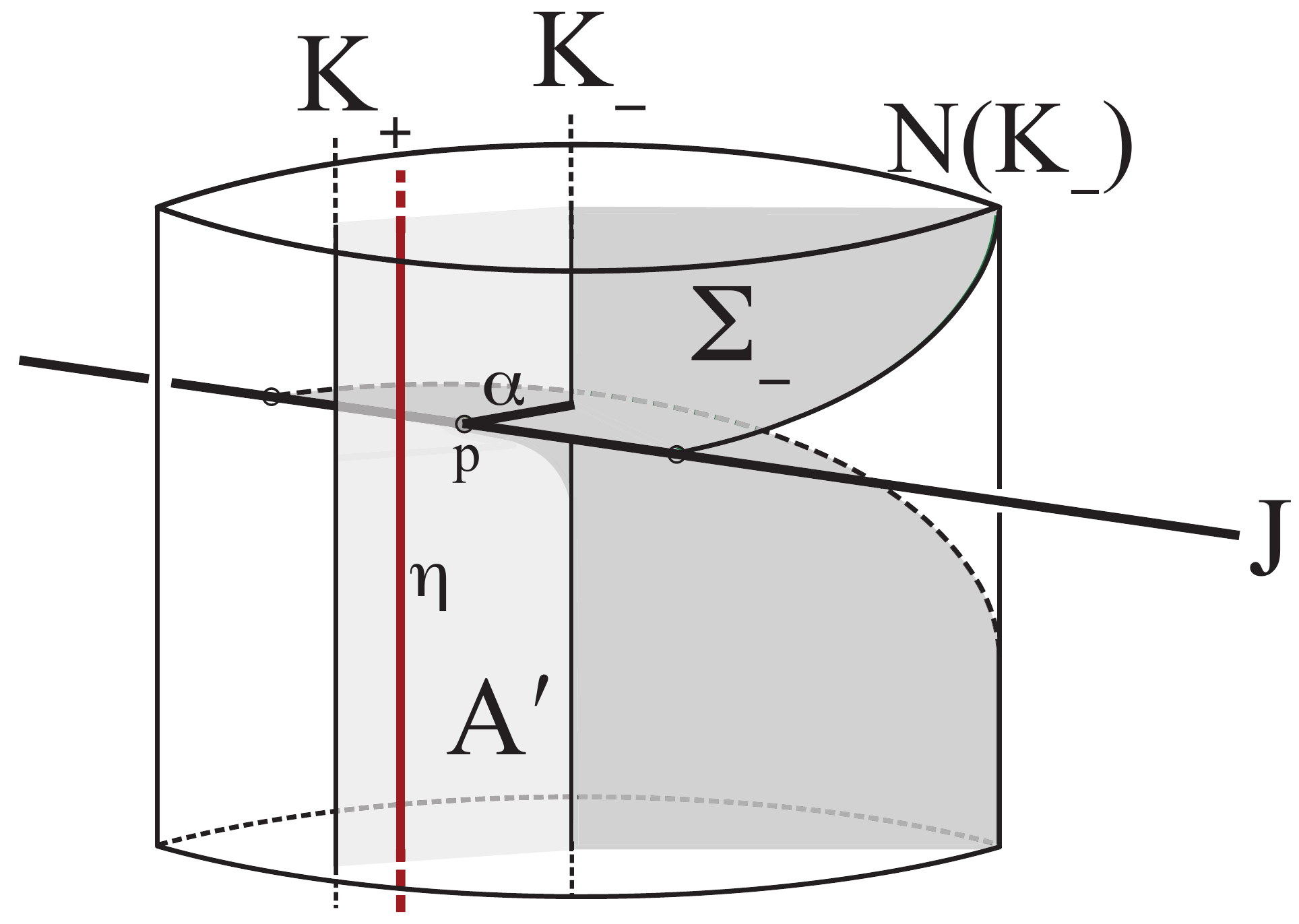}
\caption{A non-trivial circle intersection $\eta\subset \Sigma_-\cap A$.}\label{f4}
\end{figure}

Take a framed neighborhood $N(A)$ of $A$ so that $K_-$ is the inverse image of the core under a diffeomorphism from $N(A)$ t the solid torus (see \cite{geiges} or Theorem 4.1.12 in \cite{ozbagcistip}) with the framing given by $Fr_A$. For a small $\epsilon>0$, $A\subset N(K_-)$ so $K_+\subset N(K_-)$ is a parallel copy of $K_-$. If there are multiple non-trivial circles, let $\eta$ denote the one ``closest'' to $K_-$ (as in Figure \ref{f4}). Let $\Sigma_-'$ denote the part of $\Sigma_-$ bounded by $K_-$ and $\eta$. It has no other boundary components so let $\Sigma_-''=\Sigma_-'\cup A'$. This is a closed oriented surface. Since $J$ is closed of complementary dimension to $\Sigma_-''$, their geometric intersection is 0. So $J$ intersects $\Sigma_-''$ at another point besides $p$, and since $J\cap A=\{p\}$, $J$ intersects $\Sigma_-'$. But $\Sigma_-$ is embedded and cannot self-intersect.

An intersection arc is a \textit{clasp} if it runs from boundary-to-interior on each surface and a \textit{singular clasp} if it runs boundary-to-boundary on one surface (Figure \ref{f5}). 

\begin{figure}[!ht]
\includegraphics[scale=.28]{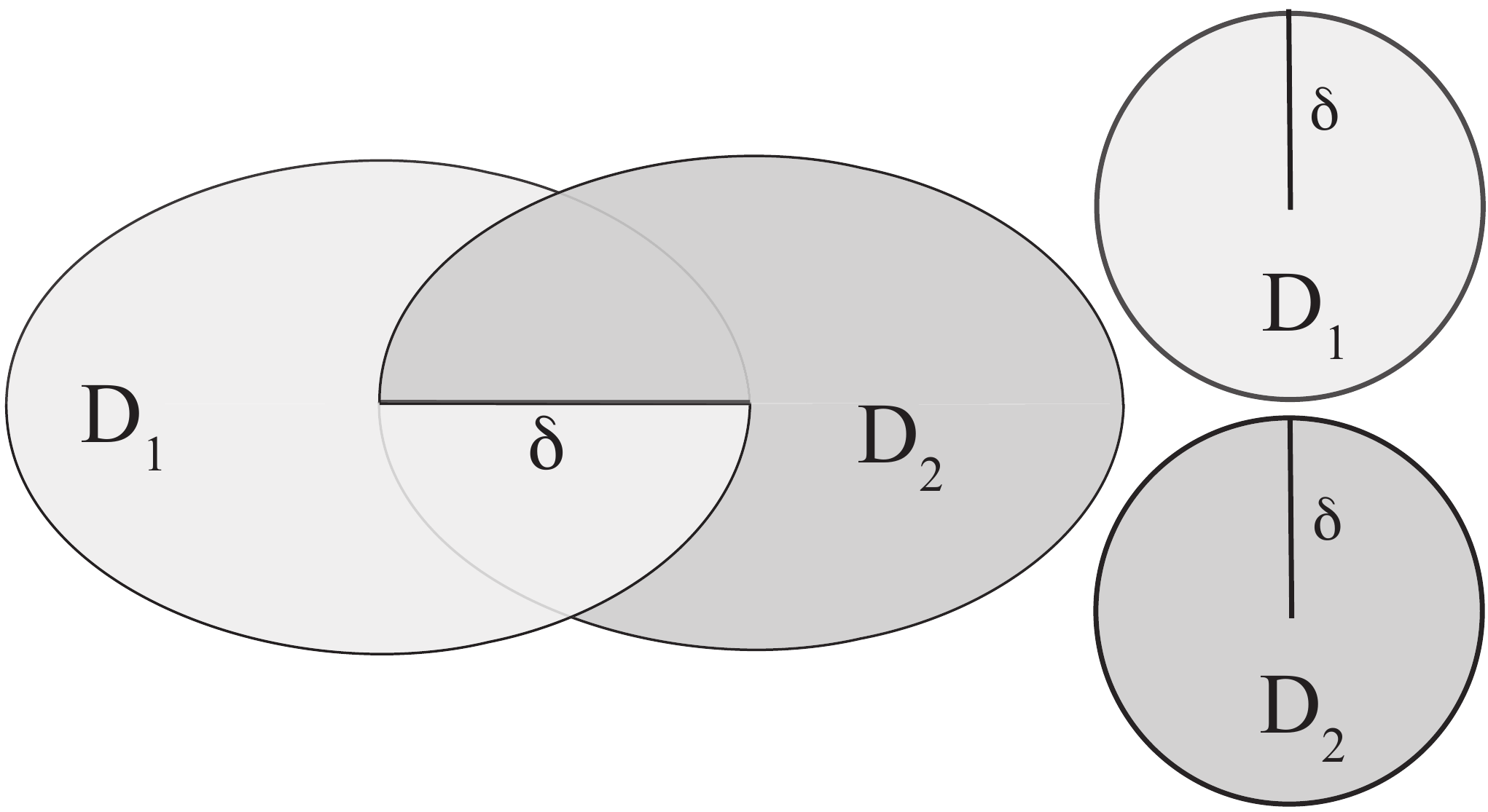}
\includegraphics[scale=.28]{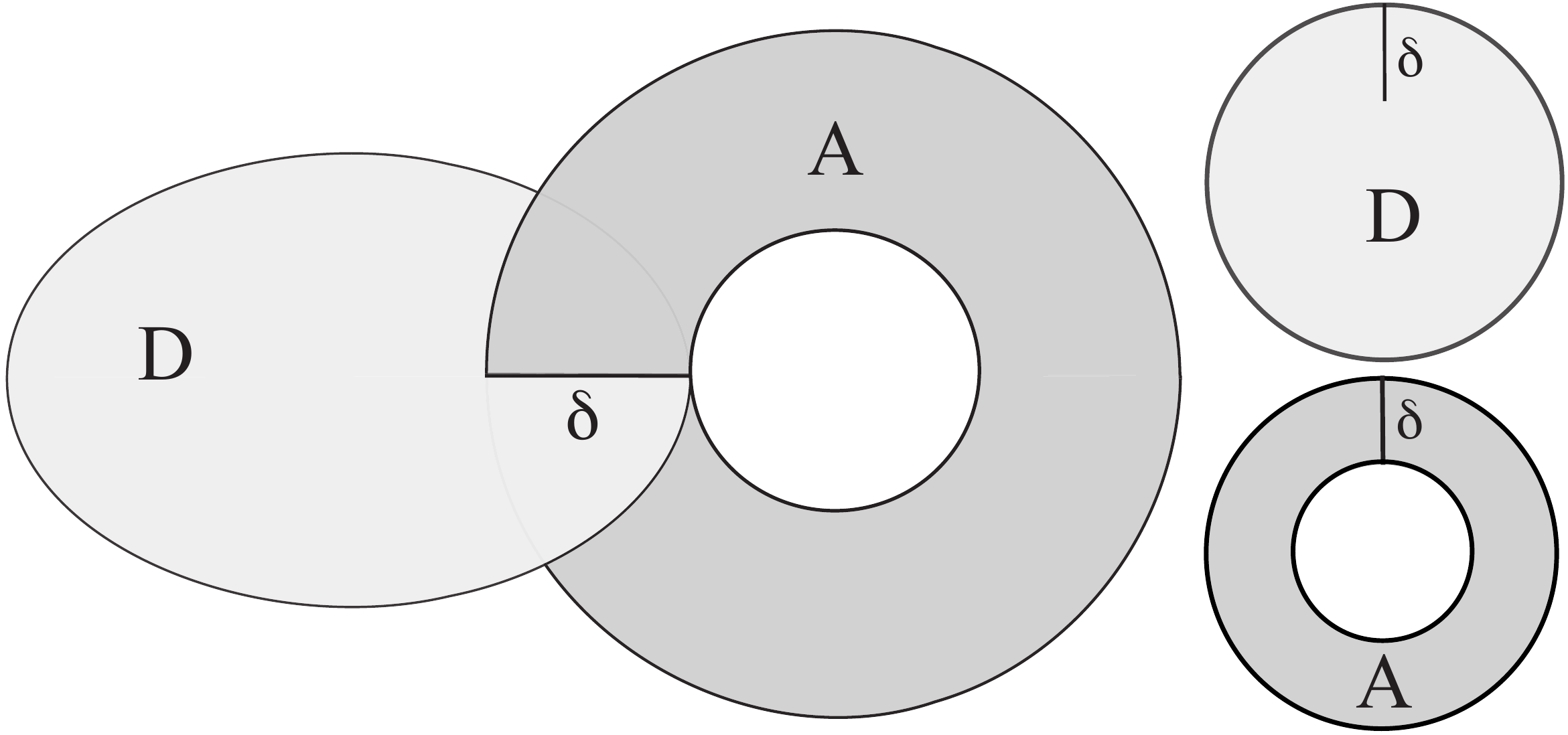}
\caption{A clasp $\delta=D_1\cap D_2$ and a singular clasp $\delta=D\cap A$.}\label{f5}
\end{figure} 

There are exactly $\rvert K_+\cdot\Sigma_-\rvert$ clasps $\beta_i$ and exactly one singular clasp $\alpha$ in $\Sigma_-\cap A$ ($\alpha$ is unique since another such arc $\alpha'$ would make $p\in\alpha\cap\alpha'$ a self-intersection of the embedded $\Sigma_-$). 

Clasps (and singular clasps) can be resolved standardly to give an embedded oriented surface (see Figures \ref{f6_0} and \ref{f6}). 

\begin{figure}[!ht]
\centering
\includegraphics[scale=.43]{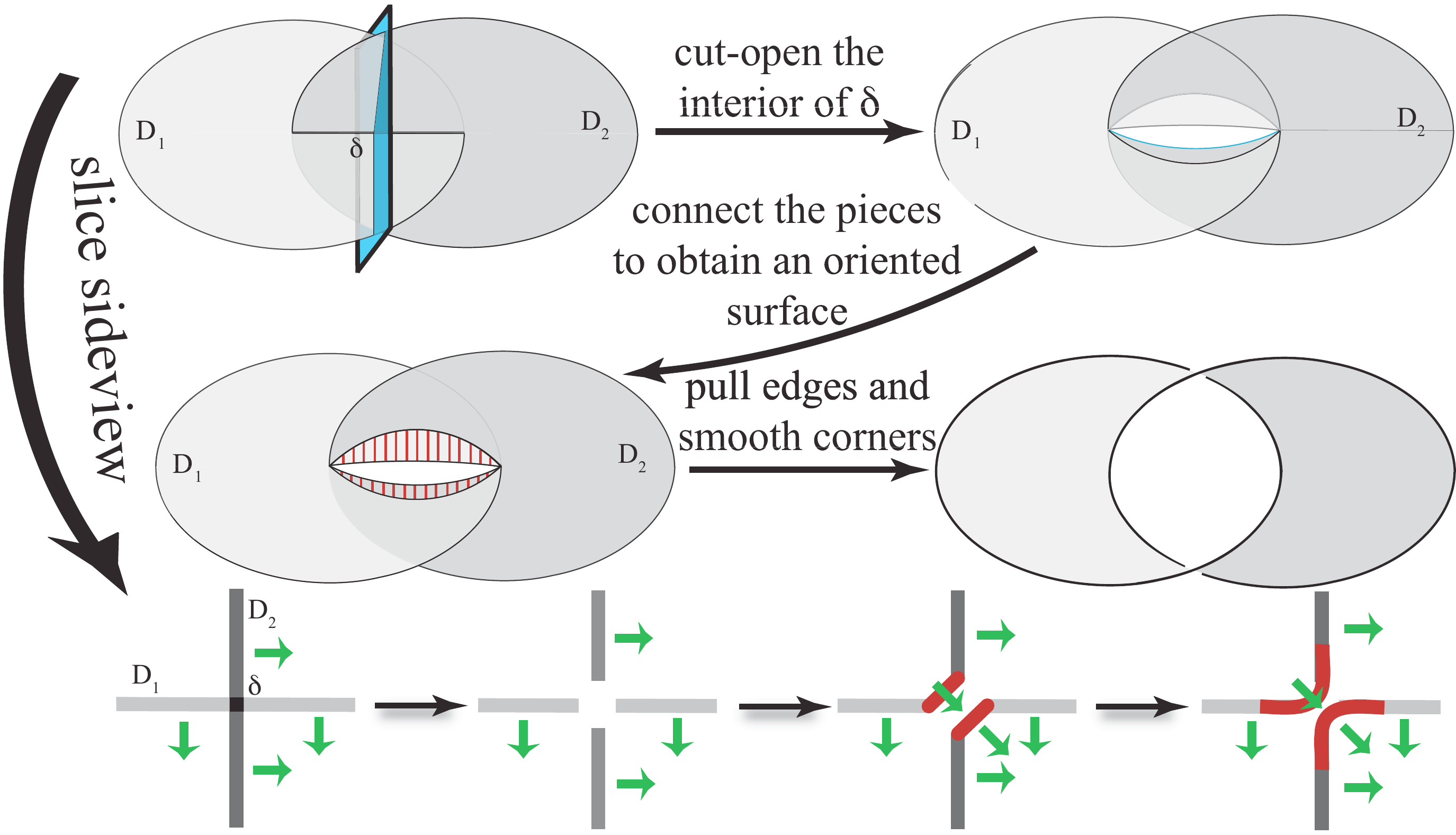}
\caption{Resolving a clasp arc. Green arrows show orientations.}\label{f6_0}
\end{figure}

It is an \textit{orientation preserving  cut} (see \cite{papi}), a \textit{smoothing of a singular arc in an orientable way}, or an \textit{Umschaltung} (M. Dehn). It consists of cutting open the interior of the clasp and joining the edges appropriately, whereby an endpoint of the clasp becomes a singular point. Cut open at this point (see Figure \ref{f6} for an example) to obtain an embedded oriented surface.

\begin{figure}[!ht]
\centering
\includegraphics[scale=.43]{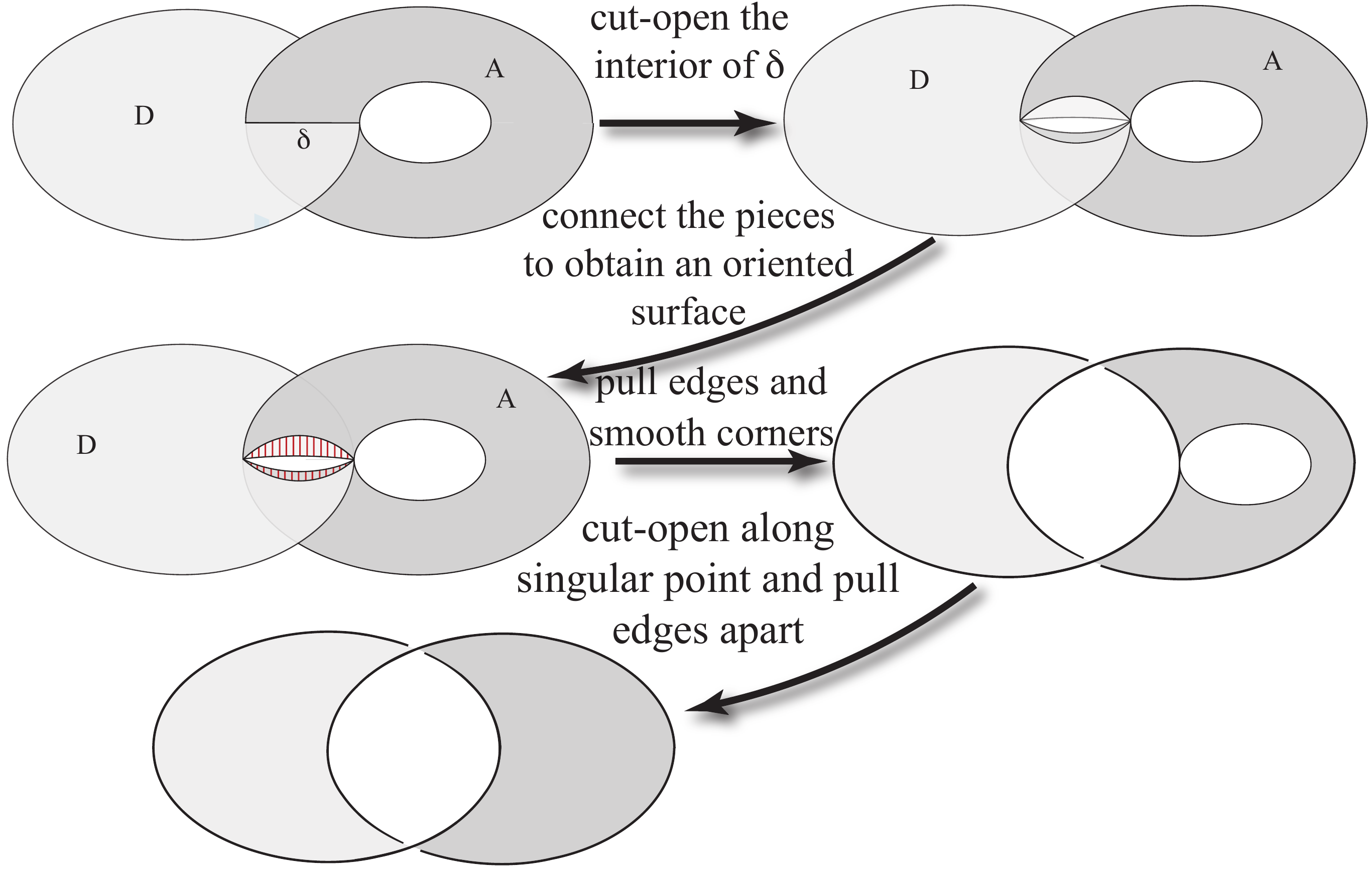}
\caption{Resolving a singular clasp arc.}\label{f6}
\end{figure}

Consider the clasps $\beta_i$. Again, we take a framed neighborhood $N(K_-)$ of $K_-$ with the framing $Fr_A$. Cut open the interiors of $\alpha$ and each $\beta_i$ (e.g., Figure \ref{f7}).

\begin{figure}[!ht]
\includegraphics[scale=.45]{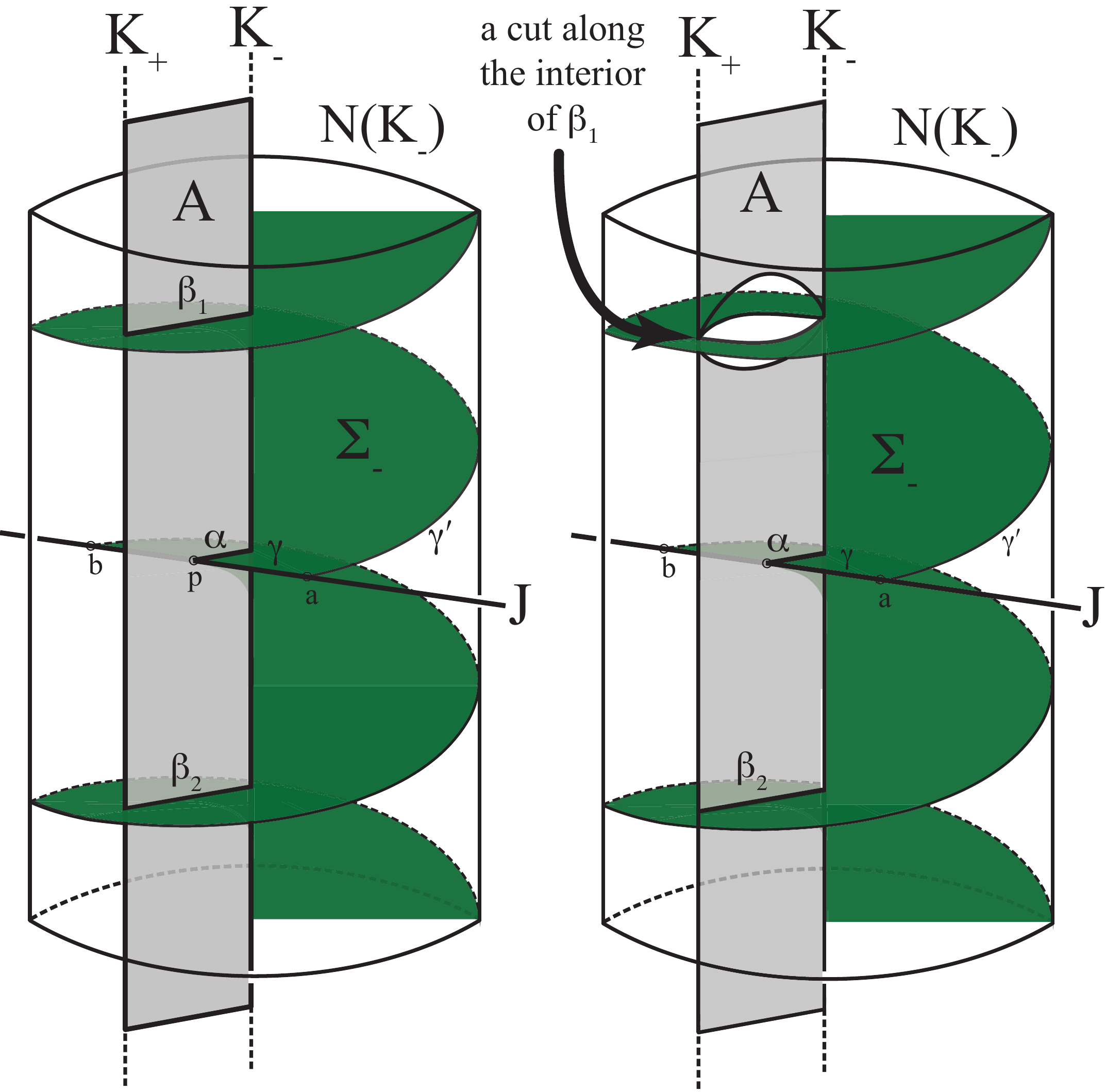}
\caption{The singular clasp $\alpha$ and the clasps $\beta_i$ with a cut along $\beta_1$.}\label{f7}
\end{figure}

Next join the edges accordingly (e.g., see Figure \ref{f8}) and cut along singular points to obtain an embedded oriented surface.

\begin{figure}[!ht]
\includegraphics[scale=.43]{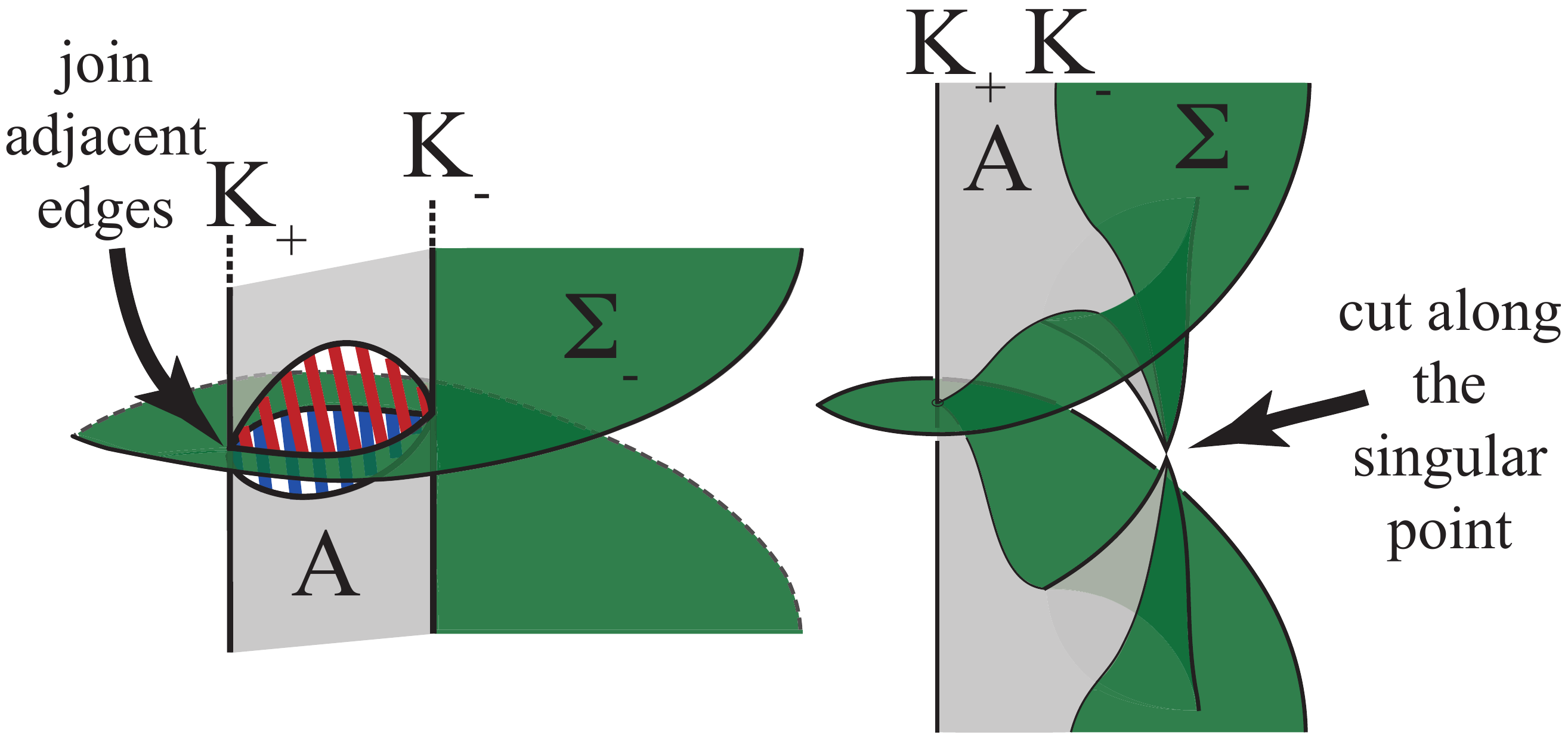}
\caption{Joining the edges and eliminating the singular point after a cut along the clasp $\beta_1$.}\label{f8}
\end{figure}

More precisely, two edges of a cut-open clasp are \textit{adjacent} if, given any two neighborhoods $U_{\Sigma_-}$ and $U_A$ of these edges in the respective surface, $K_-\cap(U_{\Sigma_-}\cap U_A)$ is a nonempty segment along $K_-$, properly containing the common boundary point of the two edges. Near two adjacent edges, the surfaces $\Sigma_-$ and $A$ are also ``adjacent''. We resolve the clasp arcs $\beta_i$ and the singular clasp $\alpha$ by joining adjacent edges (look at Figures \ref{f7} and \ref{f8}).

This produces an embedded oriented surface $\Sigma_+$ with $\partial\Sigma_+=K_+\cup J$.
\end{proof}

In the construction of Lemma \ref{construction} above, resolving all other possible intersections besides the clasps required only isotoping the interior of $\Sigma_-$ locally, so we may assume that such intersections do not occur. 

\begin{lemma}\label{twistingnearbeta} 
Consider a standard framed neighborhood $N(A)$ with the product framing on $K_-$ induced by $Fr_A$ as in the construction of Lemma \ref{construction} prior to resolving the clasp intersections between $\Sigma_-$ and $A$. Denote by $Fr_{N(K_-)}$ the product framing  induced on the core $K_-$ and its parallel translation $K_+$. Then
$$tw_{K_-}(Fr_{\Sigma_-},Fr_{N(K_-)})=tw_{K_+}(Fr_{\Sigma_+},Fr_{N(K_-)})+J\cdot A.$$
\end{lemma} 

\begin{proof}
Consider a push-off $K_-'$ of $K_-\subset\partial\Sigma_-$ into $\Sigma_-$ giving the Seifert framing $Fr_{\Sigma_-}$ along $K_-$ so that $tw_{K_-}(Fr_{\Sigma_-},Fr_{N(K_-)})=lk_A(K_-',K_-)=K_-'\cdot A$. The intersection $K_-'\cap A$ consists of a point on each clasp arc $\beta_i\subset\Sigma_-\cap A$ and a point on the singular clasp arc $\alpha\subset\Sigma_-\cap A$. Now, isotop $K_-'$ in $\Sigma_-$ through each arc $\beta_i$ across $K_+$ such that the intersection points along $\beta_i$ are eliminated and $K_-'$ intersects $A$ once along $\alpha$. This implies that away from a neighborhood of $\alpha$, $K_-'$ links with $K_-$ relative to the product framing $Fr_{N(K_-)}$ in the same way that it links with $K_+$ relative to the product framing $Fr_{N(K_-)}$ because $K_-'$ is disjoint from the annulus $A$ away from $\alpha$. 

Resolving the $\beta_i$ modifies the interior of $\Sigma_-$ away from $K_-'$. 
Resolving $\alpha$ modifies $K_-'$ into $K_-''$, which coincides with $K_-'$ away from a neighborhood of $\alpha$ and does not link with $K_+$ rel product framing, we obtain $K_-''$ from $K_-'$ by removing a small segment and adding another one, each isotopic to the other rel boundary in the complement of $K_+$.

Observe that since $\Sigma_+$ is constructed from $A\cup\Sigma_-$ by resolving their intersections, $K_-''$ is a push-off of $K_+$ into $\Sigma_+$. Relative to the product framing, it links with $K_+$ in the same way as $K_-'$ linked with $K_-$ except for the $J\cdot A$ which is an additional twist and contributes $\pm 1$.

So with respect to the product framing, $lk(K_-',K_-)-lk(K_-'',K_+)=J\cdot A$.
\end{proof}

\begin{lemma}\label{twistingnearalpha} 
In the construction of Lemma \ref{construction}, $$tw_{J}(Fr_{\Sigma_+},Fr_{\Sigma_-})=J\cdot A.$$
\end{lemma}

\begin{proof}
Resolving the clasps $\beta_i$ changes the interior of $\Sigma_-$ away from $J$ so it does not affect the framing at $J$. We need to see what changes occur to the framing $Fr_{\Sigma_-}$ along $J$ when we resolve the singular clasp $\alpha$. Prior to resolving it, take (the closure of) a standard framed tubular neighborhood $N(J)\cong S^1\times D^2$ with framing $Fr_{\Sigma_-}$ via a push-off $J_-=S^1\times \{\frac{1}{2},0\}$ into $\Sigma_-$. Consider an arc $\delta\subset A$ that ends at $p$ such that $\delta\cap N(J)=\{0\}\times [0,1]$. After $\alpha$ is resolved, $J\subset \partial \Sigma_+$, take a push-off $J'$ of $J$ into $\Sigma_+$. Since $\Sigma_-$ and $\Sigma_+$ are isotopic away from a neighborhood of the resolved arc $\alpha$, we can choose that $J'$ and $J_-$ coincide away from $p$. Observe that $J'$ intersects $\delta$, moreover, $J'\cdot \delta=J\cdot A$ by the construction of $\Sigma_+$. The winding number of the normal to $J$ defined by $J_-$ under the diffeomorphism to $N(J)$ equals 0, and with the winding number of the normal to $J$ defined by $J'$, we get $J'\cdot \delta=J\cdot A$. This gives $tw_{J}(Fr_{\Sigma_+},Fr_{\Sigma_-})$.
\end{proof}

\begin{remark}\label{alphatwist}
Let $\Sigma_+$ be the surface constructed in Lemma \ref{construction}. Then $\Sigma_+$ is smoothly isotopic to a surface obtained from $\Sigma_-$ by the following local operation in a neighborhood $D_\alpha$ of the singular clasp $\alpha$ (Figure \ref{f9}). Cut $D_\alpha$ along its boundary arcs that run across $K_-$ to $J$ in the interior of $\Sigma_-$. Then create an oriented twist in $D_\alpha$  depending on the sign of $J\cdot A=\{p\}$, as shown Figure \ref{f9}, and re-glue back by the identity diffeomorphism. This gives the diffeomorphism type of $\Sigma_+$ as constructed in Lemma \ref{construction}.

\begin{figure}[!ht]
\includegraphics[scale=.4]{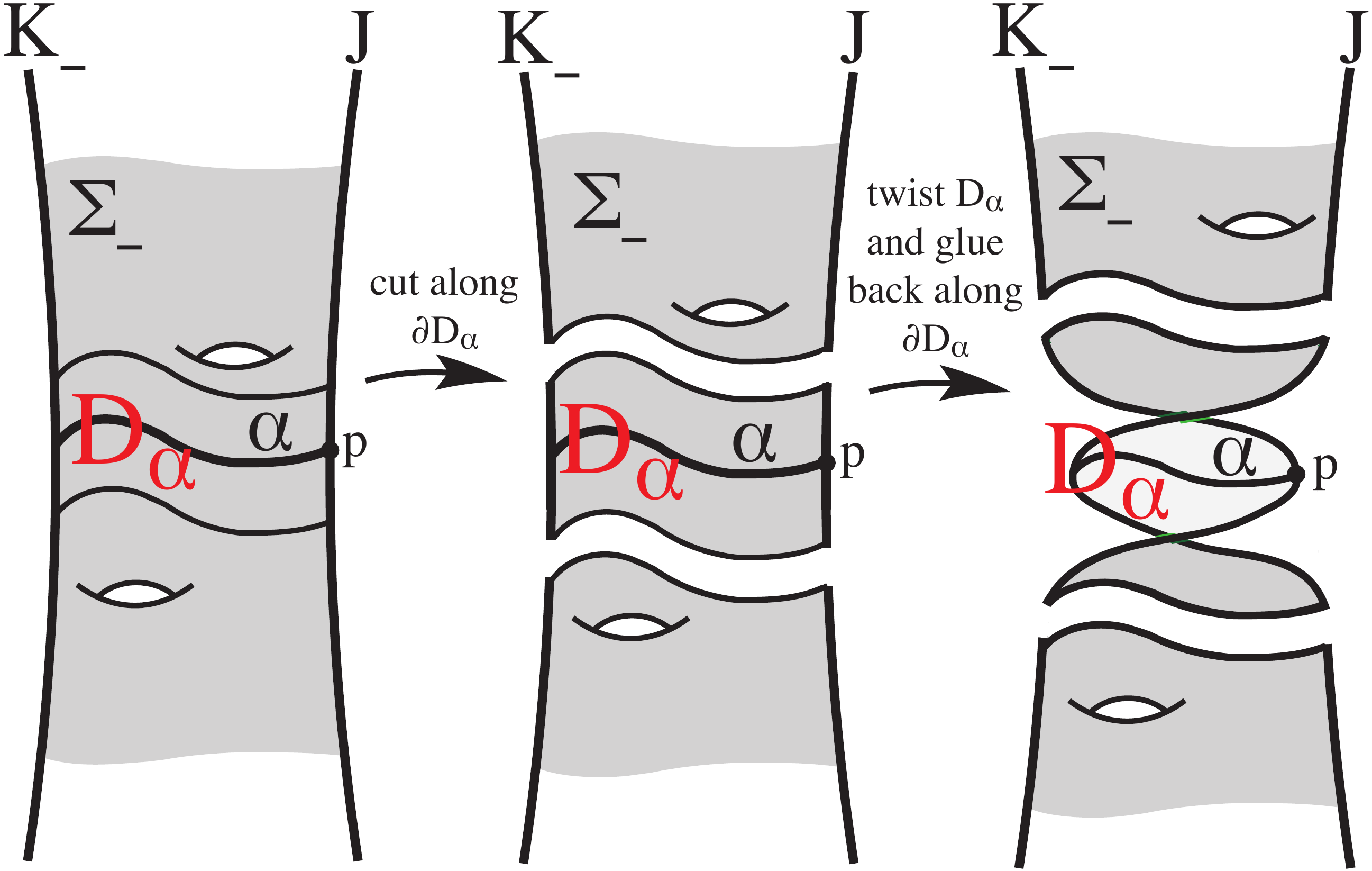}
\caption{Obtaining the diffeomorphism type of $\Sigma_+$ by a local twist in $\Sigma_-$ around the singular clasp $\alpha$. Since only $p\in\partial\alpha$ is fixed, the diffeomorphism type of $\Sigma_+$ depends on the arc $\alpha$. The twist is determined by the sign of the intersection of $J$ and $A$ at $p$.}\label{f9}
\end{figure}

To see this, note that resolving the clasp intersections $\beta_i$ creates a surface which is smoothly isotopic to $\Sigma_-\cup A$ since away from the singular clasp $\alpha$, $\Sigma_-\cup A$ deformation retracts to $\Sigma_-$. So the diffeomorphism type of $\Sigma_+$ is determined only by the resolution of $\alpha$. Now, since $K_-$ and $K_+$ are oriented as boundary components of the oriented annulus $A$, make sure that the orientations on $K_-$ as a boundary component of $\Sigma_-$ and a boundary component of $A$ coincide, consider $J\cdot A=\pm 1$. If $J\cdot A=+1$, then in the construction of Lemma \ref{construction}, resolving $\alpha$ produces an oriented surface which locally is obtained from $\Sigma_-$ by a left twist in the disc $D_\alpha$. Similarly, if $J\cdot A=-1$, the twist on $D_\alpha$ is right. Thus the diffeomorphism type of $\Sigma_+$ depends on the arc $\alpha$ and the sign of $J\cdot A$. In particular, if $\alpha$ and $\alpha'$ are two isotopic arcs (with common boundary point $p\in J$) in $\Sigma_-$ running from $K$ to $p\in J$, then the resulting surfaces $\Sigma_+$ and $\Sigma_+'$ obtained by resolving $\alpha$ and $\alpha'$ are isotopic.
\end{remark}

\section{Legendrian and Transverse Isotopy Invariance}

Consider homologous Legendrian (resp. transverse) knots $K$ and $J$ oriented as $K\cup J=\partial\Sigma$ for an oriented embedded surface $\Sigma$ in a contact 3-manifold $(M,\xi)$ so that $[\partial\Sigma]=[K]-[J]$ and a Legendrian (resp. transverse) isotopy $\varphi_t:S^1\times[0,1]\rightarrow(M,\xi)$ of $K$ in the complement of $J$, which fixes $J$. We can extend it to all of $M$ by the Isotopy Extension Theorem (in the Legendrian or transverse category, resp., see \cite{traynor}) with $\varphi_t=id$ on a neighborhood of $J$. Because $\varphi_t$ extends to an isotopy of the Seifert surface $\Sigma$, the relative invariants are preserved. That is, for a Legendrian isotopy $\varphi_t$, $\widetilde{tb}_\Sigma(\varphi_t(K),\varphi_t(J))=\widetilde{tb}_\Sigma(K,J)$ and $\widetilde{r}_\Sigma(\varphi_t(K),\varphi_t(J))=\widetilde{r}_\Sigma(K,J)$, and for a transverse isotopy $\varphi_t$, $\widetilde{sl}_\Sigma(\varphi_t(K),\varphi_t(J))=\widetilde{sl}_\Sigma(K,J)$ for all $t$.

Now consider a Legendrian (resp. transverse) isotopy which, as in Lemma \ref{construction}, isotops $K$ across $J$, that is, $\varphi_s(K)$ crosses $J$ for some $s\in (0,1)$. Then $\varphi_t$ no longer extends to an isotopy of $\Sigma$. For $\epsilon>0$, let $K_\pm$ and $\Sigma_-$ be as in Lemma \ref{construction} and assume the trace of $\varphi_t$ and transverse to $J$ at $p=\varphi_s(K)\pitchfork J$,  we will say that $\varphi_t$ is {\em locally transverse} to $J$ (see \cite{georgi} for a discussion of the case when $\varphi_t$ is not locally transverse to $J$). We claim that in this case the relative invariants are well-defined for all $t\neq s$ and are indeed invariant under $\varphi_t$.

\begin{theorem}\label{legendriantransverseinvariance} 
Consider homologous Legendrian (resp. transverse) knots $K$ and $J$ oriented as boundary components of an oriented embedded surface $\Sigma$ in a contact 3-manifold $(M,\xi)$ so that $[\partial\Sigma]=[K]-[J]$. Consider a Legendrian (resp. transverse) isotopy $\varphi_t:S^1\times[0,1]\longrightarrow(M,\xi)$ of $K$, and suppose that $\varphi_s(K)\pitchfork J$ and $\varphi_t(K)\cap J=\emptyset$ for $t\neq s$. 
\begin{enumerate}[(a)]
\item If $\varphi_t$ is Legendrian, then $\widetilde{tb}_{\Sigma'} (\varphi_t(K), J)=\widetilde{tb}_\Sigma(K,J)$ for all $t\neq s$, where $[\Sigma']\in H_2(M,\varphi_t(K)\cup J;\Z)$.
\item If $\varphi_t$ is Legendrian, then $\widetilde{r}_{\Sigma'} (\varphi_t(K), J)=\widetilde{r}_\Sigma(K,J)$ for all $t\neq s$, where $[\Sigma']\in H_2(M,\varphi_t(K)\cup J;\Z)$.
\item If $\varphi_t$ is transverse, then $\widetilde{sl}_{\Sigma'} (\varphi_t(K), J)=\widetilde{sl}_\Sigma(K,J)$ for all $t\neq s$, where $[\Sigma']\in H_2(M,\varphi_t(K)\cup J;\Z)$.
\end{enumerate}
\end{theorem}

\begin{proof}[Proof of Theorem \ref{legendriantransverseinvariance}(a)] 
We will use the construction of $\Sigma_+$ in the case of smooth isotopies (Lemma \ref{construction}). Resolving some intersections of $\Sigma_-\cup A$ involved isotoping the interior of $\Sigma_-$ locally, which does not affect the relative invariants. Resolving the clasp intersections, however, may. We assume that we have resolved all intersections besides the (possibly multiple) clasps $\beta_i$ and the singular clasp $\alpha$.

Assume first that $\varphi_t$ is Legendrian and let $\Sigma_+$ be constructed as in the proof of Lemma \ref{construction}. First, we prove the invariance of the relative Thurston-Bennequin number.

Prior to resolving any of the clasp intersections between $\Sigma_-$ and $A$, take a framed Legendrian neighborhood $N(K_-)$ of $K_-$ such that the product framing $Fr_{N(K_-)}$ is given by $Fr_A$ as in Lemma \ref{twistingnearbeta}, and so $K_+$ is a parallel copy of the Legendrian core $K_-$.

Since $K_-$ and $K_+$ are Legendrian isotopic, $tw_{K_-}(\xi,Fr_A)=tw_{K_+}(\xi,Fr_A)$, equivalently, $\widetilde{tb}_A(K_-,K_+)=0$. The reason for this is the fact that there exists an immersed annulus $A$ (the trace of the isotopy) that satisfies this, so by Lemma \ref{independent}, the embedded surface $A$ satisfies this condition as well. Equivalently, since we are in a framed Legendrian neighborhood $N(K_-)$, $tw_{K_-}(\xi,Fr_{N(K_-)})=tw_{K_+}(\xi,Fr_{N(K_-)})$ or $\widetilde{tb}_{N(K_-)}(K_-,K_+)=0$. By Lemma \ref{independent}, we could use $\Sigma_+$ to prove the invariance of the relative Thurston-Bennequin number. By the discussion at the beginning of this section, it is sufficient to prove that
$$\widetilde{tb}_{\Sigma_-}(K_-,J)=\widetilde{tb}_{\Sigma_+}(K_+,J),$$
which, by definition, is equivalent to
$$tw_{K_-}(\xi,Fr_{\Sigma_-})-tw_J(\xi,Fr_{\Sigma_-})=tw_{K_+}(\xi,Fr_{\Sigma_+})-tw_J(\xi_,Fr_{\Sigma_+})$$ 
or
\begin{equation}\label{eq1}
tw_{K_+}(\xi,Fr_{\Sigma_+})-tw_{K_-}(\xi,Fr_{\Sigma_-})=tw_J(\xi,Fr_{\Sigma_+})-tw_J(\xi,Fr_{\Sigma_-}).
\end{equation}

Since $tw_{K_\pm}(\xi,Fr_{\Sigma_\pm})=tw_{K_\pm}(\xi,Fr_{N(K_-})+tw_{K_\pm}(Fr_{N(K_-},Fr_{\Sigma_\pm})$, the left-hand side of equation (\ref{eq1}) equals
$$tw_{K_+}(\xi,Fr_{N(K_-)})+tw_{K_+}(Fr_{N(K_-)},Fr_{\Sigma_+})-tw_{K_-}(\xi,Fr_{N(K_-)})-tw_{K_-}(Fr_{N(K_-)},Fr_{\Sigma_-}).$$

Since $K_+$ is Legendrian isotopic to $K_-$, $tw_{K_-}(\xi,Fr_{N(K_-)})=tw_{K_+}(\xi,Fr_{N(K_-)})$, so the above expression simplifies to
$$tw_{K_+}(Fr_{N(K_-)},Fr_{\Sigma_+})-tw_{K_-}(Fr_{N(K_-)},Fr_{\Sigma_-}).$$
Since $tw_{K_-}(Fr_{N(K_-)},Fr_{\Sigma_-})=tw_{K_+}(Fr_{N(K_-)},Fr_{\Sigma_+})+J\cdot A$ by Lemma \ref{twistingnearbeta}, the expression above (and the left-hand side of equation (\ref{eq1})) equals $-J\cdot A$.

Now, the right-hand side of equation (\ref{eq1}) equals 
$$tw_J(\xi,Fr_{\Sigma_-})+tw_J(Fr_{\Sigma_-},Fr_{\Sigma_+})-tw_J(\xi,Fr_{\Sigma_-}),$$ 
which simplifies to $tw_J(Fr_{\Sigma_-},Fr_{\Sigma_+})$. By Lemma \ref{twistingnearalpha}, $tw_J(Fr_{\Sigma_-},Fr_{\Sigma_+})=J\cdot A$, which implies that $tw_J(Fr_{\Sigma_-},Fr_{\Sigma_+})=-J\cdot A$, and equation (\ref{eq1}) holds.
\end{proof}

\begin{proof}[Proof of Theorem \ref{legendriantransverseinvariance}(b)] 

We set some general notation. For a contact 3-manifold an a Legendrian knot bounding an oriented surface and oriented as the boundary of the surface in the 3-manifold, let $w_{\sigma_{surface}}(v_{knot})$ denote the winding number of the positive tangent vector field to the knot measured as a winding number of a loop of vectors under the identification with $\R^2$ at each point of the knot given by the trivialization $\sigma_{surface}$ of the contact structure over the surface restricted to the boundary. In the standard sense, $w_{\sigma_{surface}}(v_{knot})$ is the rotation number of the Legendrian knot. 

We start with the following simple observation.
\begin{remark}\label{rotnumber}
Consider a Legendrian unknot $K_1\subset(B^3,\xi_{std}\rvert_{B^3})$ in standard contact 3-ball around the origin in $(\R^3,\xi_{std}=\ker(dz-ydx))$.  The rotation number of $K_1$ (with respect to the global trivialization on $(B^3,\xi_{std}\rvert_{B^3})$ given by the non-zero section $\partial/\partial y$ of $\xi_{std}\rvert_{B^3}$) is given by the degree of the tangent vector to the Lagrangian projection $\pi(K_1)$, where $\pi:\R^3\rightarrow\R^3:(x,y,z)\mapsto (x,y)$ (see \cite{etnyre:knots}). Note that since $K_1$ bounds a 2-disc in $B^3$, we can consider its oriented projection under $\pi$ and compute the winding number of its oriented boundary $\pi (K_1)$, in the Lagrangian projection, the positively oriented parts of the 2-disc point in the positive $z$-direction of the $xy$-plane, so the rotation number computation amounts to counting the number of positively oriented (upward) circles in the projection of the 2-disc. Now take another oriented Legendrian knot $K_2$ which co-bounds an annulus $\Lambda$ with $K_1$, so its Lagrangian projection links with $K_1$ (positively if $K_2\cdot D=1$ and negatively if $K_2\cdot D=-1$). Orient $\Lambda$ so that the orientation that $K_1$ inherits as its boundary is the opposite to the one it inherits as boundary of $D$. Then $\xi_{std}$ is trivial over $\Lambda$, and $r_\Lambda(K_1)=w_{\sigma_{_\Lambda}} (v_{K_1})$ . Let $\sigma_{std}$ denote the trivialization from $\partial/\partial y$ on $\xi_{std}$ restricted to $K_1$, then
$$w_{\sigma_{std}}(v_{K_1})=w_{\sigma_{_{D}}}(v_{K_1})=w_{\sigma_{_\Lambda}}(v_{K_1})-K_2\cdot D.$$
This explicitly tells us how the rotation number of $K_1$ is computed for another trivialization coming from another Seifert surface, purely in terms of a linking number with a knot which defines the new Seifert framing for $K_1$.
This observation is easy to see by working with Legendrian knots ``with corners" in $(\R^3,\xi_{std})$, note that it holds in this case as well. Another approach is to use the characteristic foliations on the 2-dsic $D$ and the annulus $\Lambda$.
\end{remark}

In our situation, we have the trivializations $\sigma_{_{\Sigma_\pm}}:\xi\rvert_{\Sigma_\pm}\rightarrow\Sigma_\pm\times\R^2$ induced by each Seifert surface on $\sigma_{_{\Sigma_\pm}}$ and their restrictions to the knots $K_\pm$ and $J$. We need to show
$$\widetilde{r}_{\Sigma_-}(K_-,J)=\widetilde{r}_{\Sigma_+}(K_+,J),$$
which, by definition, is equal to
$$w_{\sigma_{_{\Sigma_-}}}(v_{K_-})-w_{\sigma_{_{\Sigma_-}}}(v_J)=w_{\sigma_{_{\Sigma_+}}}(v_{K_+})-w_{\sigma_{_{\Sigma_+}}}(v_J),$$
or, rearranging terms, 
$$w_{\sigma_{_{\Sigma_-}}}(v_{K_-})-w_{\sigma_{_{\Sigma_+}}}(v_{K_+})=w_{\sigma_{_{\Sigma_-}}}(v_J)-w_{\sigma_{_{\Sigma_+}}}(v_J).$$

We will use a local model near the intersections $\alpha$ and $\beta_i$. This is justified by the fact that away from intersection points along $K_-$, $K_+$, and $J$, we can assume that the trivializations $\sigma_{_{\Sigma_-}}$ and $\sigma_{_{\Sigma_+}}$ coincide along $\partial\Sigma_\pm$, with the special point that along $K_-$, we would need to change the sign since $K_-$ acquires opposite orientation from the surface $\Sigma_-$ and the surface $A$. Therefore, we only need to see what the local contributions to the winding numbers are near the arcs of intersection. 

First we deal with the arcs $\beta_i$, refer to Figures \ref{f7} and \ref{f_beta}.

Take an arc $\beta_i$ and consider a disc with corners $D_A\subset A$ around $\beta_i$ and boundary components $\gamma_1, \gamma_2, \eta_-, \eta_+$ with $\eta_\pm\subset K_\pm$. Isotop the $\gamma_i$ to make them Legendrian and chose the corners of $D_A$ so that the winding numbers $w_{\sigma_{_{D_A}}}(v_{\eta_\pm})$ and $w_{\sigma_{_{D_A}}}(v_{\gamma_i})$ are well-defined and $w_{\sigma_{_{D_A}}}(v_{\eta_-})=w_{\sigma_{_{D_A}}}(v_{\eta_+})$ and $w_{\sigma_{_{D_A}}}(v_{\gamma_1})=w_{\sigma_{_{D_A}}}(v_{\gamma_2})$. Now take a small 2-disc with corners $D_-\subset\Sigma_-$ around $\beta_i$ with Legendrian boundary $\delta_1\cup\zeta\cup\delta_2\cup\eta_-$, such that $w_{\sigma_{_{D_-}}}(v_{\zeta})=-w_{\sigma_{_{D_-}}}(v_{\eta_-})$ and $w_{\sigma_{_{D_-}}}(v_{\delta_1})=w_{\sigma_{_{D_-}}}(v_{\delta_2})$. Note that resolving the arc $\beta_i$ in constructing the surface $\Sigma_+$ is a local operation and we can arrange the two discs $D_A$ and $D_-$ to be contained in an arbitrarily small 3-ball neighborhood of the arc $\beta_i$. Let $S_+$ with $\partial S_+=\eta_+\cup\gamma_1\cup\gamma_2\cup\zeta$ be the new surface in our 3-ball neighborhood after the arc $\beta_i$ has been resolved. Note that $S_+$ coincides with $D_A\cup D_-$ in a neighborhood of $\gamma_1\cup\gamma_2\cup\delta_1\cup\delta_2$.

Take a contactomorphism $f$ between the 3-ball neighborhood containing $S_+$ and $(B^3,\xi_{std}\rvert_{B^3})$. Now, take 2-disc strips to complete $D_A$ to an annulus $\Lambda_A$ with Legendrian $\partial\Lambda_A=(\eta_+\cup z_+)\cup (\eta_-\cup z_-)$ and complete $D_-$ to an annulus $\Lambda_-$ with Legendrian $\partial\Lambda_-=(\eta_-\cup z_-)\cup (\zeta\cup\zeta_-)$. Arrange that $w_{\sigma_{_{\Lambda_A}}}(v_{z_+})=-w_{\sigma_{_{\Lambda_A}}}(v_{z_-})$ and $w_{\sigma_{_{\Lambda_-}}}(v_{z_-})=-w_{\sigma_{_{\Lambda_-}}}(v_{\zeta_-})$. Note also that $\eta_+\cup z_+$ bounds a 2-disc $D_1$ and that $\Lambda\Lambda_A\cup\Lambda_-\setminus (D_A\cup D_-)\cup S_+$ is an annulus which coincides with $\Lambda_A\cup\Lambda_-$ near 
$\gamma_1\cup\gamma_2\cup\delta_1\cup\delta_2$.

Now take an oriented push-off $z'$ of $\eta_-\cup z_-$ into $\Lambda_-$ and note that $f(z')\cdot f(\Lambda_A\cup D_1)=f(z')\cdot f(D_A\cup D_1)=\zeta\cdot D_1$, because $z'$ is isotopic to $\zeta\cup\zeta_-$. Also take an oriented push-off $z''$ of $\eta_+\cup z_+$ into $\Lambda$ and note that $f(z'')\cdot f(D_1)=\zeta\cdot D_1$.

\begin{figure}[!ht]
\centering
\includegraphics[scale=.4]{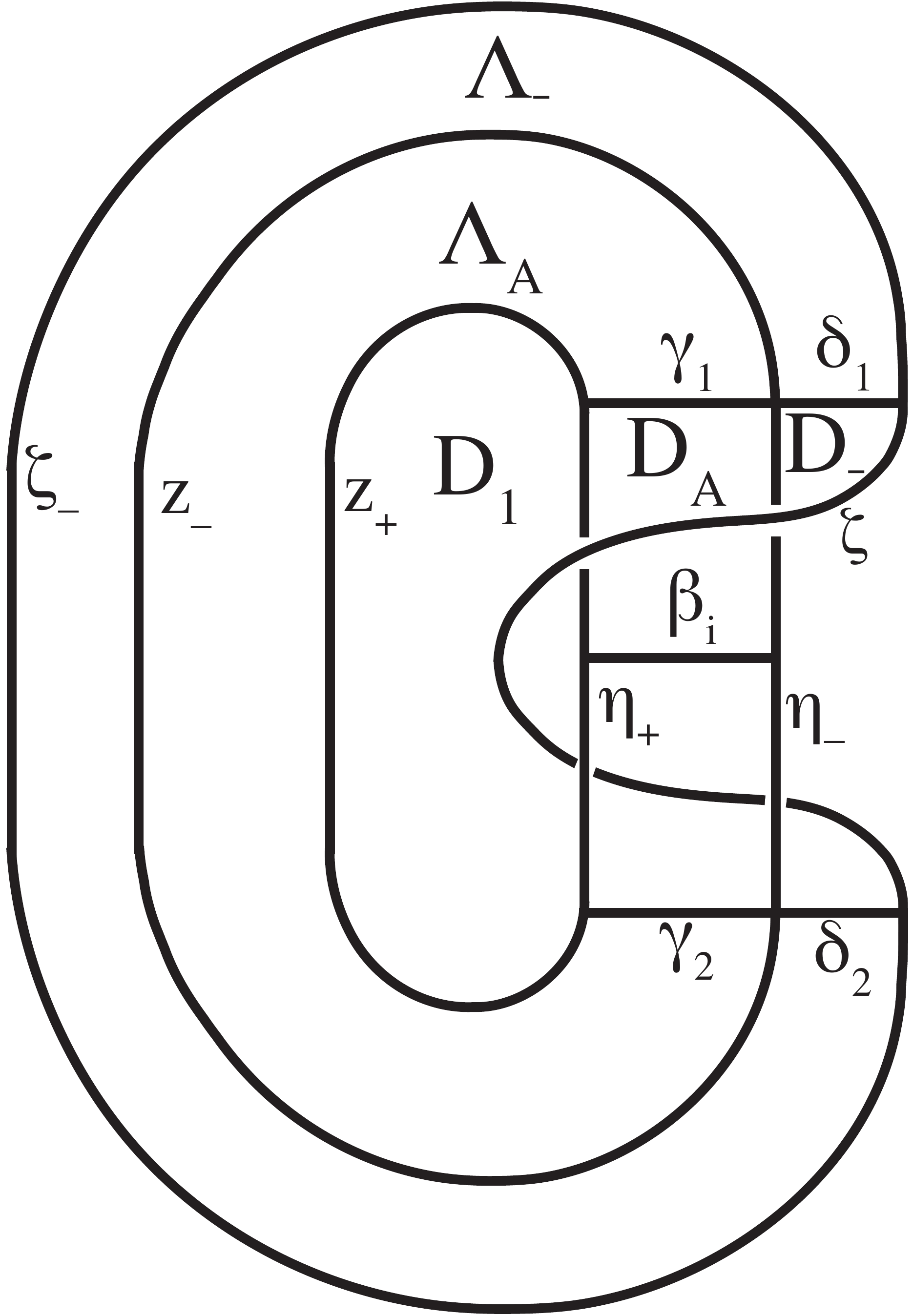}
\caption{Local picture near $\beta_i$.}\label{f_beta}
\end{figure} 

By Remark \ref{rotnumber}, 
$$w_{\sigma_{_{\Lambda_A}}}(v_{\eta_-\cup z_-})-w_{\sigma_{_{\Lambda_-}}}(v_{\eta_-\cup z_-})=-\zeta\cdot D_1$$ 
and since $w_{\sigma_{_{D_1}}}(v_{\eta_+\cup z_+})=-w_{\sigma_{_{\Lambda_A}}}(v_{\eta_+\cup z_+})$,
$$-w_{\sigma_{_{\Lambda_A}}}(v_{\eta_+\cup z_+})-w_{\sigma_{_{\Lambda}}}(v_{\eta_+\cup z_+})=-\zeta\cdot D_1.$$

Subtracting these and noting that $w_{\sigma_{_{\Lambda_A}}}(v_{\eta_+\cup z_+})=-w_{\sigma_{_{\Lambda_A}}}(v_{\eta_-\cup z_-})$, we have
$$w_{\sigma_{_{\Lambda}}}(v_{\eta_+\cup z_+})-w_{\sigma_{_{\Lambda_-}}}(v_{\eta_-\cup z_-})=0,$$
and subtracting the equal contributions of rotations along $z_-$ and $z_+$, we have
$$w_{\sigma_{_{S_+}}}(v_{\eta_+})-w_{\sigma_{_{D_-}}}(v_{\eta_-})=0,$$
which implies that as a result of the local resolution of the arc $\beta_i$, there is no contribution to the winding number of $K_+$ as compared to the winding number of $K_-$ with respect to $\Sigma_-$. Thus the winding numbers of $v_{K_-}$ and $v_{K_+}$, are fixed, regardless of which trivialization ($\sigma_{_{\Sigma_-}}$ or $\sigma_{_{\Sigma_+}}$) is used. 

Since resolving the $\beta_i$ arcs is done in the complement of $J$, we have no contribution to $w_{\sigma_{_{\Sigma_-}}}(v_J)$ and $w_{\sigma_{_{\Sigma_+}}}(v_J)$ from this, so there is no contribution to the value of the relative rotation number of $K_+$ relative to $J$ as a result of resolving the $\beta_i$.

Now we resolve the arc $\alpha$, we will clear the notation we used in resolving $\beta_i$, but stick to the notation rules we have set. We expect that it would change the individual rotation numbers of $K_+$ and $J$, but these changes would subtract out in the value of the relative rotation number. 

For the following, refer to Figure \ref{f7} and Figure \ref{f_alpha1} to understand the local picture near the arc $\alpha$. We know that $p\in\partial\alpha$ lies on $J$, let the other point in $\partial\alpha$ be $q\in K_-$. Take a half-disc $D_A\subset A$ with Legendrian boundary around $\alpha$ with corners at points $\{a,b\}\subset K_-$ and label $\partial D_A=\gamma\cup\gamma'$, where $\gamma\subset K_-$ contains the point $q$ with $\partial\gamma=\{a,b\}$ and $\alpha\setminus\{q\}$ is contained in the interior of $D_A$. 

On $\Sigma_-$, choose an arc $\eta_1$ with $\partial\eta_1=\{a,r_1\}$, where $r_1\in J$ and let the arc in $J$ bounded by $p$ and $r_1$ be $\delta_1$ so that $\gamma\cup\eta_1\cup\delta_1\cup\alpha$ bound a 2-disc $D_1\subset\Sigma_-$. Note that $D_1$ has four corners $a$, $r_1$, $p$, and $q$. Similarly, choose an arc $\eta_2$ with $\partial\eta_2=\{b,r_2\}$, where $r_2\in J$ and let the arc in $J$ bounded by $p$ and $r_2$ be $\delta_2$ so that $\gamma\cup\eta_2\cup\delta_2\cup\alpha$ bound a 2-disc $D_2\subset\Sigma_-$ with four corners $b$, $r_2$, $p$, and $q$, such that $D_1\cap D_2=\alpha$. Let $\delta=\delta_1\cup\delta_2$ and $D_-=D_1\cup D_2$. 

\begin{figure}[!ht]
\includegraphics[scale=.37]{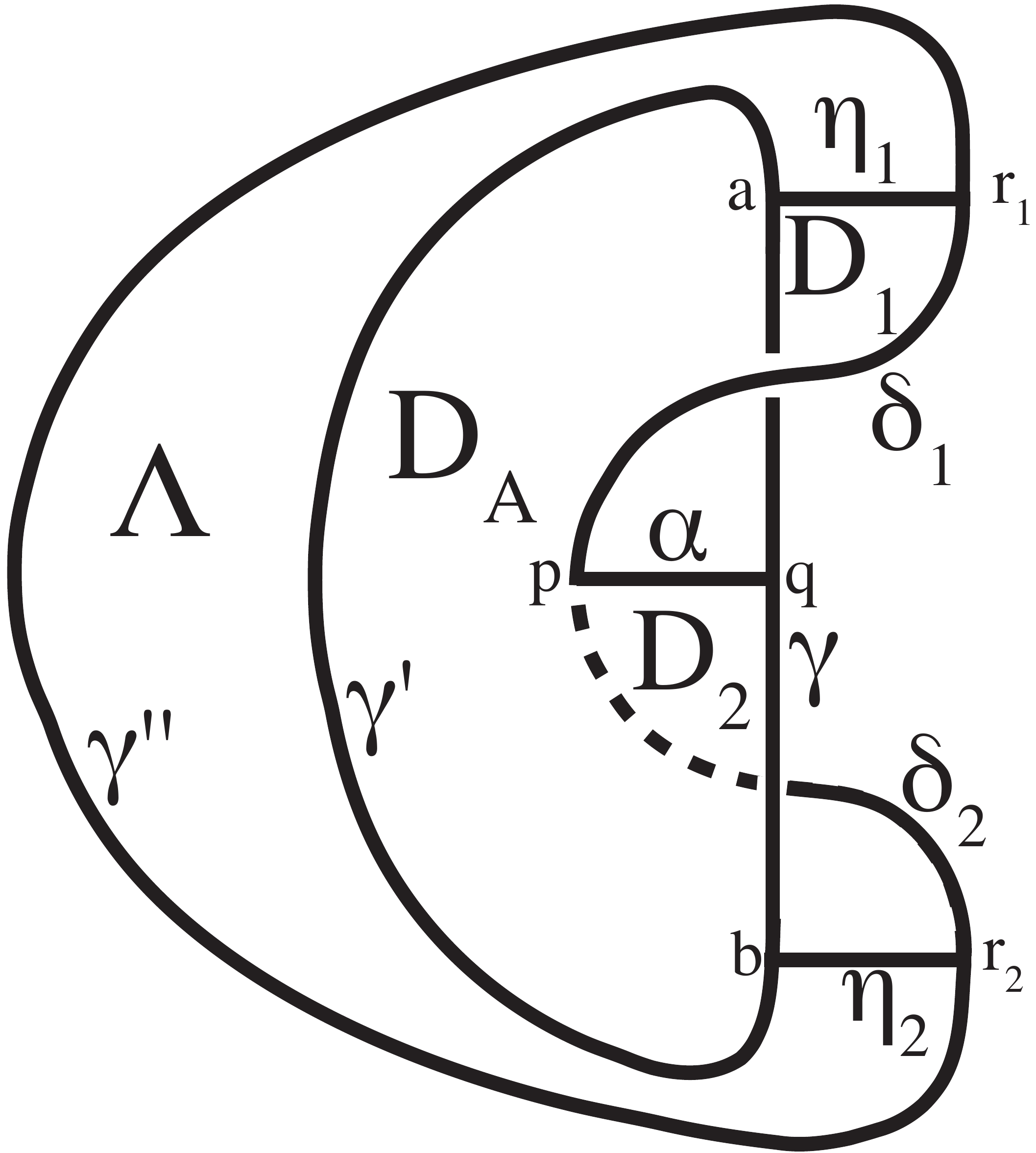}
\caption{First local picture near $\alpha$.}\label{f_alpha1}
\end{figure}

In addition, we choose a standard 3-ball neighborhood around $D_A$, $D_1$, and $D_2$ and a contactomorphism $f$ to $(B^3,\xi_{std}\rvert_{B^3})$ that will be useful later.

Since $\alpha$ is the only intersection between $\Sigma_-$ and $\Sigma_+$, assume that $\sigma_{_{\Sigma_-}}$ and $\sigma_{_{\Sigma_+}}$ coincide along $K_-$ and $J$ away from $\alpha$, in particular, assume that the two trivializations coincide along $K_-\setminus\gamma$ and $J\setminus\delta$.

By construction, 
$$(\ast)\ \ \widetilde{r}_{\Sigma_-}(K_-,J)=w_{\sigma_{_{\Sigma_-}}}(v_{K_-})-w_{\sigma_{_{\Sigma_-}}}(v_J)=w_{\sigma_{_{\Sigma_-}}}(v_\gamma)-w_{\sigma_{_{\Sigma_-}}}(v_\delta)$$

Now let $K_1=\gamma\cup\gamma'=\partial D_A$, let $K_2=\gamma''\cup\delta$, where $\Lambda$ be an annulus with Legendrian boundary $K_1\cup K_2=\partial\Lambda$. Then $f(K_1)=\partial f(D_A)$ and $f(K_1)\cup f(K_2)=\partial f(\Lambda)$, so Remark \ref{rotnumber} implies 
$$w_{\sigma_{D_A}}(v_{K_1})=w_{\sigma_{_{f(D_A)}}}(v_{f(K_1)})=w_{\sigma_{_{f(\Lambda)}}}(v_{f(K_1)})-f(K_2)\cdot f(D_A),$$  
where $f(K_2)\cdot f(D_A)=f(\delta)\cdot f(D_A)=\delta\cdot D_A=J\cdot D_A$. Since
$$w_{\sigma_{_{f(D_A)}}}(v_{f(K_1)})-w_{\sigma_{_{f(\Lambda)}}}(v_{f(K_2)})=w_{\sigma_{_{D_A}}}(v_{K_1})-w_{\sigma_{_{\Lambda}}}(v_{K_2}),$$
we obtain
$$w_{\sigma_{_{D_A}}}(v_{K_1})-w_{\sigma_{_{\Lambda}}}(v_{K_2})=-J\cdot D_A.$$
Observe that we can assume that the trivializations $\sigma_{_{D_A}}$ and $\sigma_{_{\Lambda}}$ coincide along $\gamma'$ (with opposite sign), and that we can construct $\Lambda$ so that $\gamma'$ and $\gamma''$ have equal (and opposite) contribution to the winding number of their respective boundary components. Therefore,
$$w_{\sigma_{_{D_A}}}(v_{K_1})-w_{\sigma_{_{\Lambda}}}(v_{K_2})=w_{\sigma_{_{D_A}}}(v_\gamma)-w_{\sigma_{_{\Lambda}}}(v_\gamma)=-J\cdot D_A.$$
By construction, $w_{\sigma_{_{\Lambda}}}(v_\gamma)=w_{\sigma_{_{\Sigma_-}}}(v_\gamma)$ and (with appropriate orientations)
$$w_{\sigma_{_{D_A}}}(v_\gamma)=w_{\sigma_{_{D_A}}}(v_{\gamma'})=w_{\sigma_{_{A}}}(v_{K_-})=w_{\sigma_{_{A}}}(v_{K_+})=w_{\sigma_{_{\Sigma_+}}}(v_{K_+}).$$
Thus,
$$w_{\sigma_{_{\Sigma_+}}}(v_{K_+})-w_{\sigma_{_{\Sigma_-}}}(v_\gamma)=-J\cdot D_A.$$
So in $(\ast)$, this gives
$$(\ast)\ \ \widetilde{r}_{\Sigma_-}(K_-,J)=w_{\sigma_{_{\Sigma_-}}}(v_{K_-})-w_{\sigma_{_{\Sigma_-}}}(v_J)$$
$$(\ast\ast)\ \ =w_{\sigma_{_{\Sigma_+}}}(v_{K_+})+J\cdot D_A-w_{\sigma_{_{\Sigma_-}}}(v_\delta)$$

Now to compute the term $w_{\sigma_{_{\Sigma_-}}}(v_\delta)$, we ``invert" the picture in Figure \ref{f_alpha1}, see Figure \ref{f_alpha2}. Consider the Legendrian knots $L_1=\eta_1\cup\delta\cup\eta_2\cup\gamma$ and $L_2=\eta_1\cup\delta\cup\eta_2\cup\gamma'$, then $L_1=\partial D_-$ and let $L_2=\partial D_+$ for a 2-disc $D_+$. Observe that $D_\pm$ can be isotoped so that they coincide along $\eta_i$ and so that $D_-$ and $D_A$ coincide along $\gamma'$. Also note that $w_{\sigma_{_{\Sigma_-}}}(v_{\delta})=w_{\sigma_{_{D_-}}}(v_{\delta})$.

Take a 2-disc $D_\zeta$ with corners in the complement of $D_+$  so that $\partial D_\zeta$ is Legendrian and $\partial D_\zeta=\eta_1\cup \zeta_1\cup\eta_2\cup \zeta_2$, so that it extends $D_+$ to an oriented embedded annulus $Z$ with Legendrian boundary components $\delta\cup\zeta_1$ and $\gamma'\cup\zeta_2$. Choose $D_\zeta$ so that $w_{\sigma_{_{D_\zeta}}}(v_{\zeta_1})=w_{\sigma_{_{D_\zeta}}}(v_{\zeta_2})$. Consider points $c_1,c_2\in\eta_1$ and $d_1,d_2\in\eta_2$ and take two parallel Legendrian copies $\zeta_1'$ and $\zeta_2'$ with $\partial\zeta_i'=\{c_i,d_i\}$ and such that $w_{\sigma_{_{D_\zeta}}}(v_{\zeta_1'})=w_{\sigma_{_{D_\zeta}}}(v_{\zeta_2'})=w_{\sigma_{_{D_\zeta}}}(v_{\zeta_i})$.

Now take a push-off $\delta_-$ of $\delta$ into $D_-$ and a push-off $\delta_+$ of $\delta$ into $D_+$ such that $\partial\delta_-=\{c_2,d_2\}$ and $\partial\delta_+=\{c_1,d_1\}$. Isotop $\delta_\pm$ rel boundary to make them Legendrian and let $K_1=\zeta_1'\cup\delta_+$ and $K_2=\zeta_2'\cup\delta_-$. Moreover, choose $\delta_\pm$ so that $w_{\sigma_{_{D_\pm}}}(v_{\delta_\pm})=w_{\sigma_{_{D_\pm}}}(v_\delta)$. Note that we have freedom to adjust the points $c_i$ and $d_i$ in order to ensure that the arcs satisfy these properties. 

So now we have two Legendrian unknots $K_1$ and $K_2$, and notice that $K_1$ bounds a 2-disc $D_1$ which can be isotoped in the interior to coincide with $D_+$ near $\delta$ and to coincide with $D_\zeta$ near $\zeta_1$, while $K_2$ bounds a 2-disc $D_2$ which can be isotoped in the interior to coincide with $D_\zeta$ near $\zeta_2\subset K_2$ and to coincide with $D_-$ near $\delta_-\subset K_2$. Observe that we can isotop the interior of $D_1$ so that it intersects with $D_-$ precisely along the arc $\alpha$ since $K_2$ necessarily intersects $D_1$ along the segment $\delta_-\subset K_2$.

\begin{figure}[!ht]
\includegraphics[scale=.5]{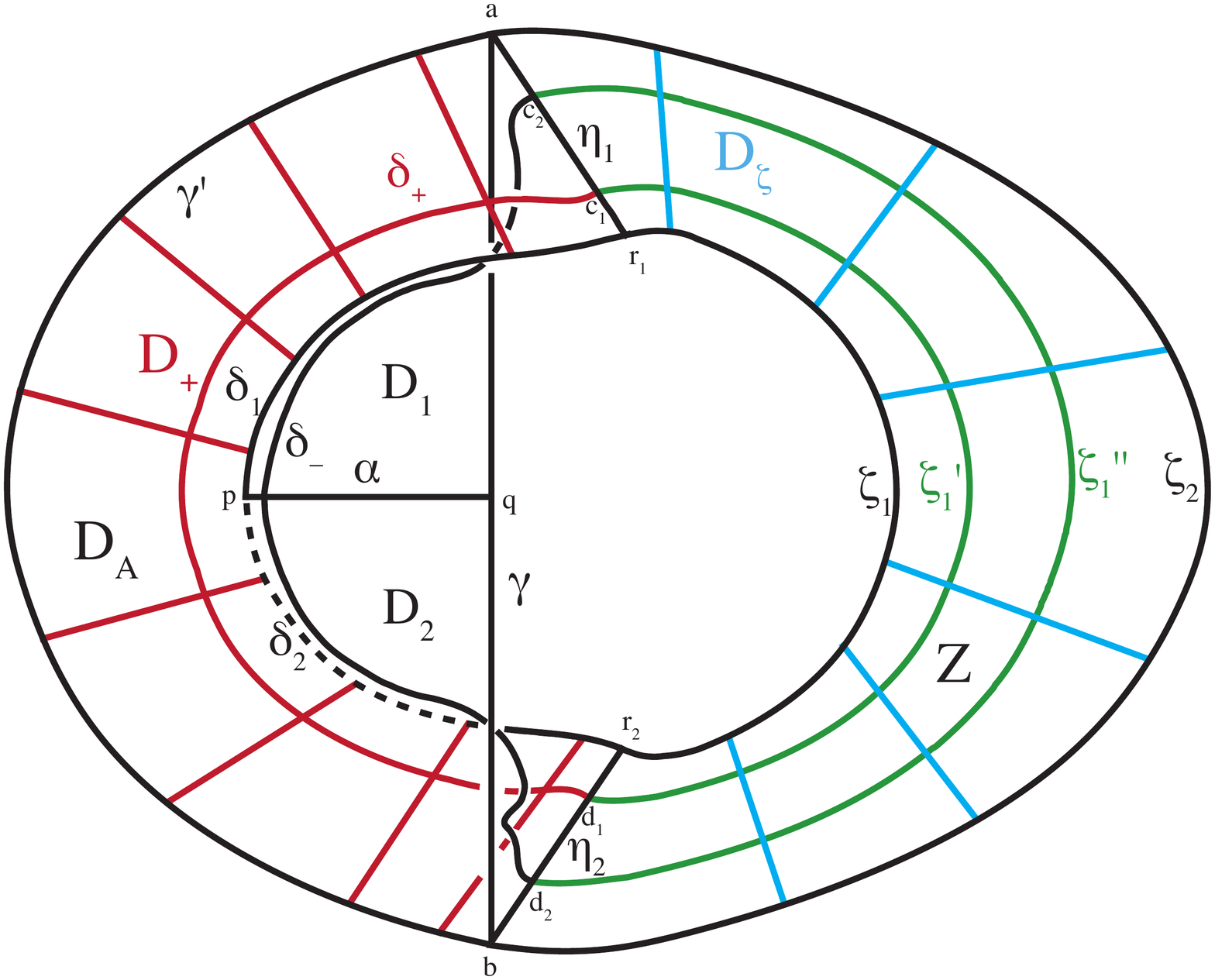}
\caption{Second local picture near $\alpha$.}\label{f_alpha2}
\end{figure}

In this setup, Remark \ref{rotnumber} applies to $K_1$ and $K_2$ to give
$$w_{\sigma_{_{D_1}}}(v_{K_1})=w_{\sigma_{_{D_2}}}(v_{K_2})-K_2\cdot D_1,$$
equivalently,
$$w_{\sigma_{_{D_1}}}(v_{K_1})-w_{\sigma_{_{D_2}}}(v_{K_2})=-K_2\cdot D_1,$$
which, after subtracting the equal contributions along the $\zeta_i'$, equals
$$w_{\sigma_{_{D_1}}}(v_{\delta_+})-w_{\sigma_{_{D_2}}}(v_{\delta_-})=-K_2\cdot D_1,$$
equivalently,
$$w_{\sigma_{_{D_+}}}(v_{\delta_+})-w_{\sigma_{_{D_-}}}(v_{\delta_-})=-K_2\cdot D_1.$$

Note that 
$$w_{\sigma_{_{D_-}}}(v_{\delta_-})=w_{\sigma_{_{\Sigma_-}}}(v_{\delta_-})=w_{\sigma_{_{\Sigma_-}}}(v_\delta),$$
so 
$$w_{\sigma_{_{\Sigma_-}}}(v_\delta)=w_{\sigma_{_{D_+}}}(v_{\delta_+})+K_2\cdot D_1.$$

By construction, 
$$w_{\sigma_{_{D_+}}}(v_{\delta_+})=w_{\sigma_{_{\Sigma_+}}}(v_{\delta_+})=w_{\sigma_{_{\Sigma_+}}}(v_\delta)=w_{\sigma_{_{\Sigma_+}}}(v_J).$$

So we have that 
$$w_{\sigma_{_{\Sigma_-}}}(v_\delta)=w_{\sigma_{_{\Sigma_+}}}(v_J)+K_2\cdot D_1.$$

Now consider the term $K_2\cdot D_1$. Look back at the 2-disc $D_A$ with $\partial D_A=\gamma\cup\gamma'$. Since the 2-disc $D_1$ coincides with $D_+$ near $\delta$, and $\alpha\subset D_+$, we can isotop the interior of $D_+$ away from $\alpha$ and $\delta$ so that $D_A\cap D_1=\alpha'$, a clasp arc such that $\alpha\subset\alpha'$.

In order to keep consistent orientations, we will need to use the discs $D_A$, $D_-$, $D_+$, and $D_1$, in addition to the annulus $\Lambda$ that $K_1$ and $K_2$ co-bound, note that $\Lambda\subset D_-\cup D_\zeta$. First, $D_A$ is given an orientation coming from $A$, and $D_+$ is given the same orientation, since it coincides with $D_A$ near $\gamma'\subset\partial D_A\cap\partial D_+$. On the other hand, by construction, we orient $D_-$ so that its common boundary segment $\gamma$ with $D_A$ receives opposite orientations as a boundary to each disc. Also, $D_1$ is oriented consistently with $D_+$ so that the strip between $\delta+$ and $\delta$ along which the two discs coincide is oriented the same way from each 2-disc. Then the annulus $\Lambda$ is oriented so that $K_1$ receives opposite orientations as an element of $\partial\Lambda$ and $\partial D_1$. This gives an orientation on $K_2$ and a sign to $K_2\cdot D_2$. Reversing the orientation on $D_A$ and carrying out this check reveals that the sign $J\cdot D_A$ does not change and the sign of $K_2\cdot D_1$ also does not change. Moreover, for each choice of orientations, an easy check reveals that $J\cdot D_A=K_2\cdot D_1$, so these two numbers are equal.  Therefore,  and we obtain
$$w_{\sigma_{_{\Sigma_-}}}(v_\delta)=w_{\sigma_{_{\Sigma_+}}}(v_J)+J\cdot D_A.$$

Then substituting in $(\ast)$ and $(\ast\ast)$, we have
$$\widetilde{r}_{\Sigma_-}(K_-,J)=w_{\sigma_{_{\Sigma_+}}}(v_{K_+})+J\cdot D_A-w_{\sigma_{_{\Sigma_-}}}(v_\delta)$$
$$=w_{\sigma_{_{\Sigma_+}}}(v_{K_+})+J\cdot D_A-(w_{\sigma_{_{\Sigma_+}}}(v_J)+J\cdot D_A)$$
$$=w_{\sigma_{_{\Sigma_+}}}(v_{K_+})-(w_{\sigma_{_{\Sigma_+}}}(v_J)$$
$$=\widetilde{r}_{\Sigma_+}(K_+,J).$$

Therefore, the construction of the new surface after a Legendrian isotopy intersects the ``reference knot" dies not change the relative rotation numbers.
\end{proof}

\begin{proof}[Proof of Theorem \ref{legendriantransverseinvariance}(c)] 
We start with the equivalent of Remark \ref{rotnumber} in the transverse category.
\begin{remark}\label{slnumber}
Consider a transverse unknot $K_1\subset(B^3,\xi_{std}\rvert_{B^3})$ in standard contact 3-ball around the origin in $(\R^3,\xi_{std}=\ker(dz-ydx))$. We can compute the self-linking number of $K_1$ with respect to the global trivialization of $\xi_{std}\rvert_{B^3}$ or with respect to a trivialization of $\xi_{std}$ over a Seifert surface (in this case, a 2-disc) for $K_1$, in which case we take a non-zero section of the trivialization and measure the linking number with $K_1$ of a push-off of $K_1$ under this section. We set some notation, let $K^{surface}$ denote the push-off of a non-zero section of $\xi\rvert_{surface}$, so, in particular, for $K_1=\partial D_1$, we have $sl_{D_1}(K_1)=lk(K^{D_1}_1,K_1)=K^{D_1}_1\cdot D_1$. Now take another oriented transverse knot $K_2$ which co-bounds an annulus $\Lambda$ with $K_1$. Orient $\Lambda$ so that the orientation that $K_1$ inherits as its boundary is the opposite to the one it inherits as boundary of $D$. Then
$$K_1^{D_1}\cdot D_1=K_1^{\Lambda}\cdot D_1-K_2\cdot D_1.$$

\end{remark}

In the case of transverse knots and a transverse isotopy between $K_-$ and $K_+$, with the setup in Lemma \ref{construction}, we need to prove $\widetilde{sl}_{\Sigma_+}(K_+,J)=\widetilde{sl}_{\Sigma_-}(K_-,J)$, which by definition is equivalent to
$$K_+^{\Sigma_+}\cdot\Sigma_+-J^{\Sigma_+}\cdot\Sigma_+=K_-^{\Sigma_-}\cdot\Sigma_--J^{\Sigma_-}\cdot\Sigma_-,$$
or,
$$K_+^{\Sigma_+}\cdot\Sigma_+-K_-^{\Sigma_-}\cdot\Sigma_-=J^{\Sigma_+}\cdot\Sigma_+-J^{\Sigma_-}\cdot\Sigma_-.$$

Again we consider the cases when resolving the arcs $\beta_i$ and the arc $\alpha$ separately. By Remark \ref{slnumber} above, the arguments follow exactly the argument in the proof of part (b). We will use the figures and labels from part (b). 

We consider a local picture near $\beta_i$ and reduce the discussion to computing the local contributions to the self-linking numbers. This is justified by working with the characteristic foliation and arranging that away from a local 2-dsic neighborhood of the intersection arcs on each surface, the contribution to the self-linking difference is zero. Use the fact that self-linking is additive under relative connected sums (see \cite{georgi}). We provide some more detail of the setup.

Again, consider the setup in Figure \ref{f_beta}, choose the non-zero section of $\xi\rvert_{A}$ and the arcs $\gamma_i$ so that along $\gamma_i$, we have no contribution to the quantity $(\gamma_1\cup\eta_-\cup\gamma_2\cup\eta_+)^{D_A}\cdot D_A$. Similarly, choose the endpoints of the arcs $\eta_\pm$ so that they have equal (and opposite in sign) contributions to $(\gamma_1\cup\eta_-\cup\gamma_2\cup\eta_+)^{D_A}\cdot D_A$. Note that this is possible since we know that $K_+^{A}\cdot A=-K_-^{A}\cdot A$. Do the same for the arcs $\delta_i$ and choose the arc $\zeta$ so that $(\delta_1\cup\eta_-\cup\delta_2\cup\zeta)^{D_-}\cdot D_-=0$.

Resolving $\beta_i$ is a local modification away from $J$, so $J^{\Sigma_+}\cdot\Sigma_+-J^{\Sigma_-}\cdot\Sigma_-=0$. We claim that $K_+^{\Sigma_+}\cdot\Sigma_+-K_-^{\Sigma_-}\cdot\Sigma_-$ is also equal to zero in an appropriate sense computed locally near $\beta_i$. Consider
$$K_+^{\Sigma_+}\cdot\Sigma_+-K_-^{\Sigma_-}\cdot\Sigma_-=K_+^{\Sigma_+}\cdot\Sigma_+-K_-^A\cdot A+K_-^A\cdot A-K_-^{\Sigma_-}\cdot\Sigma_-$$
and recall that $K_-^A\cdot A=K_+^A\cdot A$, so the above expression equals
$$(\ast)\ \ K_+^{\Sigma_+}\cdot\Sigma_+-K_+^A\cdot A+K_-^A\cdot A-K_-^{\Sigma_-}\cdot\Sigma_-.$$

Observe that $K_+^{\Sigma_+}\cdot\Sigma_+=K_+^{\Sigma_+}\cdot A+K_+'\cdot \Sigma_+$, where $K_+'$ is a push-off of $K_+$ into $A$, and, similarly, $K_-^{\Sigma_-}\cdot\Sigma_-=K_-^{\Sigma_-}\cdot A+K_-'\cdot \Sigma_-$, where $K_-'$ is a push-off of $K_-$ into $A$. Therefore, $(\ast)$ becomes
 $$K_+^{\Sigma_+}\cdot A+K_+'\cdot \Sigma_+-K_+^A\cdot A+K_-^A\cdot A-K_-^{\Sigma_-}\cdot A-K_-'\cdot \Sigma_-,$$
rearranging terms, we obtain
$$K_+^{\Sigma_+}\cdot A-K_+^A\cdot A+K_+'\cdot \Sigma_+-(K_-^{\Sigma_-}\cdot A-K_-^A\cdot A+K_-'\cdot \Sigma_-).$$
By the local construction in part (b) for the case for $\beta_i$, Remark \ref{slnumber} yields that $K_+^{\Sigma_+}\cdot A-K_+^A\cdot A=-K_+'\cdot \Sigma_+$ and $K_-^{\Sigma_-}\cdot A-K_-^A\cdot A=-K_-'\cdot\Sigma_-$. So the quantity $(\ast)$ is equal to zero.

Now for resolving $\alpha$, we again look at the expression 
$$K_+^{\Sigma_+}\cdot\Sigma_+-K_-^{\Sigma_-}\cdot\Sigma_-=J^{\Sigma_+}\cdot\Sigma_+-J^{\Sigma_-}\cdot\Sigma_-.$$

Use the setup and terminology from Figure \ref{f_alpha2} to conclude that $J^{\Sigma_+}\cdot\Sigma_+-J^{\Sigma_-}\cdot\Sigma_-=K_+^{\Sigma_+}\cdot\Sigma_+-K_-^{\Sigma_-}\cdot\Sigma_-=J\cdot D_A$.
\end{proof}

Together with Lemma \ref{construction}, Lemmas \ref{independent}, \ref{independent2}, and \ref{independent3} imply that the relative Thurston-Bennequin number and relative rotation number are Legendrian isotopy invariants, and the relative self-linking number is a transverse isotopy invariant, even when the isotopy of $K$ intersects the ``reference knot" $J$.

\section{A Few Remarks}

In the case when our contact manifold $(M,\xi)$ is tight, we can use convex surface theory and the characteristic foliation of the surfaces (see \cite{etnyre:intro, giroux:convex}) to compute the relative invariants and apply them to relative connected sums and classifications (see \cite{georgi}).

In \cite{chantraine}, Chantraine looked at Legendrian knots that are cobordant via a Lagrangian cylinder in and conjectures an explicit formula which essentially gives the relative Thurston-Bennequin invariant in this construction.

In \cite{georgi2}, we study the more general setup of cobordant Legendrian knots with no restrictions on the embedded surface in the cobordism, and prove a relative slice genus bound. 

\section{Relative Framings of Transverse Knots}

In \cite{ch1, ch2}, Chernov defines and proves well-definedness of relative self-linking numbers using homotopy methods and transverse homotopy instead of isotopy. This requires certain assumptions on the contact structure of the 3-manifold (tightness, or, equivalently, coorientability). The relative Thurston-Bennequin invariant and the relative self-linking number we have defined here behave like the affine self-linking invariant in Theorem 2.0.2. parts 1 and 2a in \cite{ch2}.

\end{document}